\documentclass[11pt,reqno]{amsart}

\usepackage[T1]{fontenc}
\usepackage[boldsans]{ccfonts}
\usepackage{eulervm}
\renewcommand{\mathbf}{\mathbold}

\usepackage{enumerate, amsmath, amsthm, amsfonts, amssymb,  
  mathrsfs, graphicx, paralist}
\usepackage[usenames, dvipsnames]{color}
\usepackage[margin=1in]{geometry} 
\usepackage{comment}

\usepackage{tikz-cd}

\numberwithin{equation}{section}
\newtheorem{Theorem}[equation]{Theorem}
\newtheorem{Proposition}[equation]{Proposition}

\newtheorem{Conjecture}[equation]{Conjecture}

\theoremstyle{definition}
\newtheorem{Remark}[equation]{Remark}

\newtheorem{eg}[equation]{Example}

\newcommand{\bB}{\mathbf{B}}

\newcommand{\cF}{\mathcal{F}}

\newcommand{\bG}{\mathbf{G}}
\newcommand{\cG}{\mathcal{G}}
\newcommand{\fG}{\mathfrak{G}}

\newcommand{\cK}{\mathcal{K}}

\newcommand{\cM}{\mathcal{M}}

\newcommand{\cN}{\mathcal{N}}

\newcommand{\cO}{\mathcal{O}}

\newcommand{\cP}{\mathcal{P}}

\newcommand{\bT}{\mathbf{T}}

\newcommand{\bU}{\mathbf{U}}
\newcommand{\cU}{\mathcal{U}}

\newcommand{\cZ}{\mathcal{Z}}

\newcommand{\fb}{\mathfrak{b}}

\newcommand{\fg}{\mathfrak{g}}

\newcommand{\fl}{\mathfrak{l}}

\newcommand{\fp}{\mathfrak{p}}

\newcommand{\fs}{\mathfrak{s}}

\newcommand{\fu}{\mathfrak{u}}

\newcommand{\CC}{\mathbb{C}}

\newcommand{\EE}{\mathbb{E}}
\newcommand{\FF}{\mathbb{F}}
\renewcommand{\FF}{\mathrm{Frac}}
\newcommand{\GG}{\mathbb{G}}
\newcommand{\HH}{\mathbb{H}}

\newcommand{\OO}{\mathbb{O}}

\newcommand{\ZZ}{\mathbb{Z}}

\renewcommand{\phi}{\varphi}
\renewcommand{\emptyset}{\varnothing}

\renewcommand{\tilde}[1]{\widetilde{#1}}

\makeatletter
\def\Ddots{\mathinner{\mkern1mu\raise\p@
\vbox{\kern7\p@\hbox{.}}\mkern2mu
\raise4\p@\hbox{.}\mkern2mu\raise7\p@\hbox{.}\mkern1mu}}
\makeatother

\newcommand{\Hom}{\operatorname{Hom}}

\DeclareMathOperator{\sgn}{sgn}

\newcommand{\GL}{\mathrm{GL}}
\newcommand{\SL}{\mathrm{SL}}

\newcommand{\Gr}{\cG r}

\newcommand{\suchthat}{\mid}

\newcommand{\textif}{\text{ if }}
\newcommand{\textand}{\text{ }\mathrm{and}\text{ }}

\renewcommand{\sl}{\fs\fl}
\newcommand{\gl}{\fg\fl}

\newcommand{\tcN}{\tilde{\cN}}
\newcommand{\tfg}{\tilde{\fg}}
\newcommand{\eu}{\mathrm{eu}}
\newcommand{\pt}{\mathrm{pt}}
\newcommand{\del}{\partial}
\newcommand{\e}{\tilde{e}}
\newcommand{\f}{\tilde{f}}

\renewcommand{\top}{\mathrm{top}}
\newcommand{\Gro}{\mathbf{Gro}}
\newcommand{\Irr}{\mathrm{Irr}}
\newcommand{\E}{\EE}
\newcommand{\dd}{{\underline{d}}}

\newcommand{\height}{\mathrm{ht}}

\newcommand{\terminology}[1]{{\bf #1}}

\begin{document}

\title[Weight zero Mirkovi\'c-Vilonen basis]{Weyl group action on weight zero Mirkovi\'c-Vilonen basis and equivariant multiplicities}
\author{Dinakar Muthiah}
\address{Kavli Institute for the Physics and Mathematics of the Universe (WPI), The University of Tokyo Institutes for Advanced Study, The University of Tokyo, Kashiwa, Chiba 277-8583, Japan}
\email{dinakar.muthiah@ipmu.jp}
\maketitle

\begin{abstract}
  We state a conjecture about the Weyl group action coming from Geometric Satake on zero-weight spaces in terms of equivariant multiplicities of  Mirkovi\'c-Vilonen cycles. We prove it for small coweights in type A. In this case, using work of Braverman, Gaitsgory and Vybornov, we show that the Mirkovi\'c-Vilonen basis agrees with the Springer basis.  We rephrase this in terms of equivariant multiplicities using work of Joseph and Hotta. We also have analogous results for Ginzburg's Lagrangian construction of $\mathfrak{sl}_n$ representations. 
\end{abstract}

\section{Introduction}

The celebrated Geometric Satake equivalence of Mirkovi\'c and Vilonen \cite{Mirkovic-Vilonen} provides a geometric realization of the category of representations of a reductive group. Their work is a foundational result in the geometric Langlands program, and they work with arbitrary coefficients, so their work is intimately connected with modular representation theory. Finally, of most relevance to us, Mirkovi\'c and Vilonen construct a basis in every irreducible representation (we work with $\CC$-coefficients) indexed by the so-called Mirkovi\'c-Vilonen (MV) cycles. The combinatorial study of such cycles via the theory of MV polytopes has progressed quite far (see e.g. \cite{Kamnitzer}), but understanding of the MV basis remains fairly murky (see \cite{Baumann} for some general progress in this direction and \cite{Braverman-Gaitsgory-Vybornov} for a type A description that plays a critical role in this paper).

Even when compared with other geometric constructions of representations, many of which involve difficult intersection theory, the Geometric Satake equivalence is more mysterious.  Unlike many other geometric constructions, the action is not via an explicit description of Chevalley generators. Rather it follows from abstract facts about Tannakian categories. Vasserot \cite{Vasserot-2002} has succeeded in describing the action of the Chevalley generators in general, but even his description has some mystery: the $E$ generators are constructed via a recipe involving the first Chern class of an ample line bundle, but the construction of the $F$-generators involves an appeal to the $\sl_2$-action of the Hard Lefschetz theorem.

\subsubsection{Main conjecture and theorem}
Let us briefly recall some details about the Mirkovi\'c-Vilonen basis. Fix a complex reductive group $\bG$. For each dominant coweight $\lambda$ of $\bG$ (the coweight $\lambda$ is a \emph{weight} for the Langlands dual group $\bG^\vee$), one can construct a variety $\overline{\Gr^\lambda}$. Under the Geometric Satake equivalence the intersection homology of $\overline{\Gr^\lambda}$ is canonically identified with the irreducible module $V(\lambda)$ of $\bG^\vee$ with highest weight $\lambda$. Additionally, for any coweight $\mu$ (not necessarily dominant) Mirkovi\'c and Vilonen consider an infinite-dimensional variety denoted $S_\mu$. They show that under Geometric Satake the $\mu$-weight space $V(\lambda)_\mu$ is canonically identified with the top dimensional compactly supported cohomology $H^\top_c(\overline{\Gr^\lambda} \cap S_\mu)$.

We will be interested in the case of $\mu = 0$, which corresponds to the zero weight space $V(\lambda)_0$.  We have an action of the Weyl group $W$ on this zero weight space, which is realized as $H^\top_c(\overline{\Gr^\lambda} \cap S_0)$. It will be more convenient to consider the $W$-action on the  dual space $H_\top(\overline{\Gr^\lambda} \cap S_0)$, which is the top-degree Borel-Moore homology.

A basis of $H_\top(\overline{\Gr^\lambda} \cap S_0)$ is given by fundamental classes of the irreducible components of $\overline{\Gr^\lambda} \cap S_0$. The space $\overline{\Gr^\lambda} \cap S_0$ is invariant under a maximal torus $\bT$ of $\bG$, and it contains a unique fixed point $0$. Therefore, for each irreducible component
$Z$ it makes sense to speak about $e^\bT_0(Z)$, which is the $T$-equivariant multiplicity of $Z$ at $0$. The $T$-equivariant multiplicity exactly gives the contribution at $0$ when expressing the fundamental class $[Z]$ using the localization theorem (see \S \ref{sec:localization-theorem}).

We have $e^\bT_0(Z) \in \FF(H^\bullet_\bT(\pt))$, which is the fraction field of the $\bT$-equivariant cohomology of a point. 
Observe that the Weyl group $W$ acts on $\FF(H^\bullet_\bT(\pt))$. A main goal of ours is to advance the following conjecture.

\begin{Conjecture}[Conjecture \ref{conj:equivariant-multiplicities-and-weight-zero}   in the main text]
  \label{conj:main-conjecture}
  The map
  \begin{align}
    \label{eq:114}
H_\top(\overline{\Gr^\lambda} \cap S_0) \rightarrow     \FF(H^\bullet_\bT(\pt))
  \end{align}
  defined on the MV basis by
  \begin{align}
    \label{eq:115}
   [Z] \mapsto e^T_0(Z) 
  \end{align}
for each $Z \in \Irr(\overline{\Gr^\lambda} \cap S_0)$ is $W$-equivariant.
\end{Conjecture}
\noindent Our main theorem is the following.
\begin{Theorem}[Theorem \ref{thm:main-theorem-in-text} in the main text]
  Conjecture \ref{conj:main-conjecture} is true for $\bG = \SL_d$ and $\lambda \leq d \omega_1$.
\end{Theorem}

\subsection{Description of our work}

Our primary method for making contact with the MV basis is work \cite{Braverman-Gaitsgory-Vybornov} of Braverman, Vybornov, and Gaitsgory. Their work, in type A, realizes the MV basis via an action on the top cohomology of big Spaltenstein varieties (see \S \ref{sec:big-spaltenstein-varieties} for the terminology).
Their construction proceeds by realizing Schur-Weyl duality and symmetric $(\GL_n,\GL_n)$-duality geometrically.
This action on big Spaltenstein varieties was first considered in \cite{Braverman-Gaitsgory} by Braverman and Gaitsgory. We proceed by carefully analyzing this construction.

\subsubsection{Warmup with Ginzburg's construction}

The main result of the original work of Braverman and Gaitsgory  \cite{Braverman-Gaitsgory} is an explicit realization of Ginzburg's Lagrangian construction of $\sl_n$ representations on little Spaltenstein varieties (see \S \ref{sec:little-spaltenstein-varieties} for the terminology). Specifically, it differs from the action on big Spaltenstein varieties by a sheaf-theoretic Fourier transform. The Braverman-Gaitsgory construction gives a sheaf-theoretic interpretation of the Ginzburg action (which is a priori convolution-theoretic) on little Spaltenstein varieties.

We consider the case of small representations (for $\SL_d$, these are exactly those with highest weight $\leq d \omega_1$) and their zero weight spaces. In this case, the relevant little Spaltenstein varieties are exactly Springer fibers. We show that the action on the zero weight space is exactly given by the Springer action (Theorem \ref{thm:ginzburg-zero-weight-equals-springer}). We prove this by an explicit computation in equivariant Borel-Moore homology.

\subsubsection{The main work}

Our first task is to analyze the Braverman-Gaitsgory construction. Their construction is via sheaf theory and Schur-Weyl duality. We show that the construction can be realized by certain explicit symmetrization formulas, and we show (Theorem \ref{thm:bg-chevalley-by-explicit-correspondences}) that the Braverman-Gaitsgory action is realized via convolution by explicit correspondences. Our work described above provides a convolution-theoretic construction of the Braverman-Gaitsgory action on big Spaltenstein varieties (analogous to how Braverman and Gaitsgory provide a sheaf-theoretic construction of Ginzburg's convolution-theoretic action on little Spaltenstein varieties).

We then focus on the case of small representations and their zero weight spaces. In this case, the relevant big Spaltenstein varieties are exactly Springer components (in terms of reduced scheme structure, big and little coincide in this case). We show that the action on the zero weight space is exactly given by the Springer action tensored with the sign character (Theorem \ref{thm:bgv-weight-zero-springer-action}). We prove this by an explicit computation in equivariant Borel-Moore homology, which is exactly analogous to the computation for little Spaltenstein varieties.

Finally, to phrase our main result, we need to move from Springer components to the so-called orbital varieties. Joseph conjectured \cite{Joseph} and Hotta proved \cite{Hotta} that the Springer representation can be realized using orbital varieties and equivariant multiplicities. However, to use the Joseph-Hotta construction, we need to verify that two \emph{a priori} different ways of identifying MV cycles and orbital varieties agree (one is due to Braverman, Vybornov, and Gaitsgory, and the comes from a classical construction of Lusztig). We show this is true by carefully unwinding the work of  Braverman, Vybornov, and Gaitsgory and comparing it with the usual correspondence between Springer components and orbital varieties. This is the contents of \S \ref{sec:MV-cycles-and-orbital-varieties} and specifically Theorem \ref{thm:the-two-bijections-of-irreducible-components-agree}.

\subsubsection{A remark about existing work comparing Geometric Satake and Springer theory}

The relationship between weight-zero spaces of small representations and Springer theory has been studied \cite{Reeder}, and in particular the relationship between Springer theory and Geometric Satake has been studied by works of Achar, Henderson, and Riche \cite{Achar-Henderson,Achar-Henderson-Riche}. However, those works do not make any mention of how the Mirkovi\'c-Vilonen basis or of how equivariant multiplicities fit into their story.  Nonetheless, we suspect there should a way to see our results from their perspective. As they work with small representations in general, perhaps one can use their point of view to extend the main theorem to small representations in general.

Finally, we should mention Mautner's work \cite{Mautner}, which geometrically realizes Schur-Weyl duality on the affine Grassmannian. This is implicit in the Braverman-Gaitsgory-Vybornov construction.

\subsection{Acknowledgements}

I am grateful to Joel Kamnitzer, Allen Knutson, Ivan Mirkovi\'c, Alexei Oblomkov, and Alex Weekes for many enlightening conversations. I also thank Joel Kamnitzer, Eric Sommers, and Hiraku Nakajima for their comments on an earlier version of this manuscript.

\section{Notation and preliminaries}

Given vectors $v_1,\ldots,v_\ell$ is a vector space, we will write $\langle v_1, \ldots, v_\ell \rangle$ for their span.
Let $Z$ be a complex algebraic variety, and let $d \geq 0$ be an integer. We write $\Irr(Z)$ for the set of irreducible components of $Z$, and we write $\Irr_d(Z)$ for the set of $d$-dimensional irreducible components of $Z$.

We will write $H_\bullet(Z)$ to denote the Borel-Moore homology of $Z$ with complex coefficients. If $d = \dim Z$, then $H_{2d}(Z)$ is the top degree non-vanishing homology group, and it is canonical isomorphic to the formal span of $\Irr_d(Z)$. We will write $H_{\top}(Z) = H_{2d}(Z)$ in this situation.  
If we have an algebraic torus $A$ acting on $Z$, we will also consider the equivariant Borel-Moore homology $H^A_\bullet(Z)$. This equivariant setting will be discussed in more detail below.

Similarly, let us write $H^\bullet(Z)$ for the cohomology of $Z$ and $H^\bullet_c(Z)$ for the compactly supported cohomology. Recall that $H^\bullet_c(Z)$ is canonically the dual vector space of $H_\bullet(Z)$. Therefore, the highest degree non-vanishing cohomology $H^\top_c(Z)$ is identified with the space of complex-valued functions on the set of top-dimensional irreducible components of $Z$. In particular, $H^\top_c(Z)$ has a basis indexed by the top-dimensional irreducible components of $Z$. 

\subsection{Some type A representation theory}
\label{sec:rep-theory}

For any positive integer $k$, we will write  $[k] = \{1,\ldots,k\}$. We fix integers $n \geq 1$ and $d \geq 1$ throughout. Let:
\begin{align}
  \label{eq:1}
\cP_{n,d} = \left\{ (d_1,\ldots,d_n) \suchthat d_i \geq 0 \textand d_1 + \cdots + d_n = d \right\}
\end{align}
We call $\cP_{n,d}$ the set of ($n$-step) \terminology{compositions of $d$}. We will write $\cP_{n,d}^{++} \subseteq \cP_{n,d}$ for the set of $(d_1,\ldots,d_n) \in \cP_{n,d}$ such that $d_1 \geq \cdots \geq d_n$. So $\cP_{n,d}^{++}$ denotes the $n$-step partitions of $d$ (with some parts allowed to be size zero). Recall the usual notation $\lambda \vdash d$ that denotes that $\lambda$ is a partition of $d$ (with no restriction on the number of parts); observe that equivalently we can write $\lambda \in \cP_{d,d}^{++}$ using our notation. We will often consider the case $n=d$. Let us write $1^d = (1,\ldots,1) \in \cP^{++}_{d,d}$ for the $d$-step composition of $d$ consisting of all $1$'s.

\subsubsection{Symmetric groups} Let $S_d$ denote the set of permutations of the set $[d]$. Let $\lambda \vdash d$, and let $S(\lambda)$ denote the corresponding Specht module  of $S_d$. Over $\CC$, the set $\left\{S(\lambda) \suchthat \lambda \vdash d \right\}$ is an enumeration of the irreducible representations of $S_d$. 

For each $\dd \in \cP_{n,d}$, we will form the subgroup $S_{\dd} = S_{d_1} \times \cdots \times S_{d_n}$ consisting of all permutations of $S_d$ that preserve the discrete flag $\emptyset \subseteq [d_1] \subseteq [d_1 + d_2] \subseteq \cdots \subseteq [d_1+\cdots+d_n] = [d]$.

\subsubsection{General/special linear groups}

We will write $G=\GL_d(\CC) $ for the general linear group, and we will write $\fg_d$ for the Lie algebra of $G$. We will also consider the special linear group $\SL_d$. Let $v_1,\ldots,v_d$ denote the standard basis for $\GL(\CC^d)$. This defines a maximal torus of $G$. Let $T$ be corresponding maximal torus of $\SL_d$. We identify $S_d$ as the Weyl group of $G$ and as the Weyl group of $\SL_d$. The distinction between these Weyl groups is slight, but it will be clear from context which copy of $S_d$ is relevant.

For each $\dd \in \cP_{n,d}$, consider the partial flag $\emptyset \subseteq \langle v_1, \ldots v_{d_1}\rangle  \subseteq \langle v_1, \ldots v_{d_1 + d_2}\rangle \subseteq \cdots \subseteq \langle v_1, \ldots v_{d_1+\cdots+d_n}\rangle = \CC^d$. Let $P_{\dd}$ be the parabolic subgroup of $G$ consisting of all elements that preserve this flag. The Weyl group of the Levi factor of $P_{\dd}$ is naturally identified with $S_{\dd}$. Write $\fp_{\dd}$ for the Lie algebra of $P_{\dd}$, and write $\fu_{\dd}$ for the Lie algebra of the unipotent radical of $P_{\dd}$. Write $\fu^-_{\dd}$ for the image of $\fu_{\dd}$ under the Chevalley involution, so $\fg_d = \fp_{\dd} \oplus \fu^-_{\dd}$. Similarly, write $\fp^-_{\dd}$ for the image of $\fp_{\dd}$ under the Chevalley involution.

Notice that $P_{1^d}$ is the standard Borel subgroup of $G$ consisting of invertible upper triangular matrices. The Lie algebra $\fp_{1^d}$ consists of upper triangular matrices, $\fu_{1^d}$ are strictly upper triangular matrices, and 
$\fu^-_{1^d}$ are strictly lower triangular matrices. We will often write $B=P_{1^d}$, $\fb = \fp_{1^d}$, $\fu = \fu_{1^d}$, and $\fu^- = \fu^-_{1^d}$.

Write $\E = \CC^n$, and let $\theta_1,\ldots,\theta_n$ denote the standard basis of $\E$. For each $i \in [n]$, let $\E_i$ denote the span of $\theta_i$. Then we can write:
\begin{align}
  \label{eq:6}
 \E = \E_1 \oplus \cdots \oplus \E_n 
\end{align}
 Some constructions we consider will only depend on this decomposition into lines (and not the individual basis vectors), and this notation will help to indicate this.

For each $\dd = (d_1,\ldots,d_n) \in \cP_{n,d}$, let
\begin{align}
  \label{eq:2}
  \E^\dd = \E_1^{\otimes d_1} \otimes \cdots \otimes \E_n^{\otimes d_n}
\end{align}
which is a distinguished line in $\E^{\otimes d}$.
We will consider the general and special linear groups $G = \GL(\E) = \GL_n$ and $\SL(\E) = \SL_n$ along  with their Lie algebras $\gl_n$ and $\sl_n$. The decomposition \eqref{eq:6} determines a maximal torus $D$ of $G$.

\subsubsection{Some basic facts about Schur-Weyl duality and irreducible modules}

Consider the $d$-fold tensor product $\E^{\otimes d}$. We have a $S_d \times \GL(\E)$ action on $\E^{\otimes d}$. For us Schur-Weyl duality is the statement that, as a $S_d \times \GL(\E)$-module, $\E^{\otimes d}$ decomposes as
\begin{align}
  \label{eq:3}
  \E^{\otimes d}  = \bigoplus_{\lambda \in \cP^{++}_{n,d} } S(\lambda^t) \otimes V(\lambda) 
\end{align}
where $S(\lambda^t)$ is a Specht module, and $V(\lambda)$ is an irreducible $\GL(\E)$-module. The decomposition implies that we can construct $V(\lambda)$ as:
\begin{align}
  \label{eq:5}
  V(\lambda) = \Hom_{S_d}(S(\lambda^t), \E^{\otimes d})
\end{align}
Notice that this decomposition also holds for the $\SL(\E)$, and we can explicitly identify the highest weights for each irreducible representation $\GL(\E)$. Let $\lambda = (\lambda_1,\cdots,\lambda_n) \in \cP^{++}_{n,d} $, then we have
$V(\lambda) = V( (\lambda_1-\lambda_2) \omega_1 + \cdots + (\lambda_{n-1} -\lambda_n)\omega_{n-1} )$ as $\SL(\E)$-modules where $\{\omega_1,\ldots,\omega_{n-1}\}$ are the fundamental weights for $\SL(\E)$. We will abuse notation and also write $\lambda$ for the $\SL(\E)$-weight $(\lambda_1-\lambda_2) \omega_1 + \cdots + (\lambda_{n-1} -\lambda_n)\omega_{n-1}$.

\subsubsection{Weyl group action on weight zero spaces}
For each $a \in [n-1]$, define:
\begin{align}
  \label{eq:116}
  T_a = \exp(E_a) \exp(-F_a) \exp(E_a) \in \SL_d
\end{align}
where $E_a,F_a \in \sl_n$ are the $a$-th simple Chevalley generators.
These operators realize the simple reflections in the Weyl group $S_n = N(T)/T$. In any integrable $\sl_n$-module, the right hand side of \eqref{eq:48} acts the same way as a finite sum in the enveloping algebra $\cU(\sl_n)$ given by expanding the exponentials and truncating sufficiently high powers. The operators $T_a$ satisfy the braid relations, but do not square to the identity element. However, the operators $T_a$ do preserve the zero weight space of any integrable module and furthermore they do square to the identity as on the zero weight space. Therefore, we obtain explicit formulas for the Weyl group action on zero weight spaces.

\subsubsection{Small representations}

We will focus on the case of $n=d$. The dominant $\SL_d$ weights satisfying $\lambda \leq d \omega_1$ are exactly the dominant weights corresponding to $\lambda \vdash d$. Following a standard terminology, we call corresponding representations $V(\lambda)$ \terminology{small representations} of $\sl_d$. Looking at non-vanishing weight spaces, one immediately sees that
\begin{align}
  \label{eq:117}
  T_a \equiv 1 - E_a F_a - F_a E_a + \frac{1}{2} E_a F^2_a E_a
\end{align}
as operators on the zero weight space of small representations.

One can further simplify formula \eqref{eq:117} using $\sl_2$ relations. Notice that $E_a F_a$ and $F_a E_a$ act the same way on the weight zero space. So we can write \eqref{eq:117} as $ 1 - 2 E_a F_a  + \frac{1}{2} E_a F_a E_a F_a$. We can further write the inner $F_a E_a$ as $E_a F_a + 2$ because that operator is acting on a $-2$-weight space for the $a$-th root $\sl_2$. Finally noting that $F_a^2$ acts by zero on this zero weight space, we conclude that:
\begin{align}
  \label{eq:122}
 T_a \equiv 1 - E_a F_a 
\end{align}
The author thanks Joel Kamnitzer for alerting him to this simplification.

\subsection{Some type A geometry}

For each $\dd = (d_1, \ldots,d_n) \in \cP_{n,d}$, we can form the partial flag variety:
\begin{align}
  \label{eq:9}
  \cF_{\dd} = \left\{ F_\bullet = (F_1,\ldots,F_n) \suchthat 0 = F_0 \subseteq F_1 \subseteq \cdots \subseteq F_n = \CC^d \text{ where } \dim F_i/F_{i-1} = d_i \right\}
\end{align}
We define:
\begin{align}
  \label{eq:12}
  \cF_{n,d} = \bigsqcup_{\dd \in \cP_{n,d}} \cF_\dd
\end{align}
Notice that for $\cF_{1^d}$ is exactly the complete flag variety consisting of complete flags in $\CC^d$.
We identify
\begin{align}
  \label{eq:10}
  \cF_\dd = G/P_{\dd}
\end{align}
but note that there are $\dd \neq \dd'$ with $P_{\dd} = P_{\dd'}$, so $\cF_{\dd} \neq \cF_{\dd'}$.

\subsubsection{Little Spaltenstein varieties}
\label{sec:little-spaltenstein-varieties}
Let us write $\cN_d \subset \fg_d$ for the nilpotent cone. The $G$-orbits on $\cN_{d}$ are indexed by partitions of $d$. For each $\lambda \vdash d$, we write $\OO_\lambda$ for the corresponding nilpotent orbit and $\overline{\OO_\lambda}$ for its closure. 

We have the closed subvariety $\cN_{n,d} \subseteq \cN_d$ defined as:
\begin{align}
  \label{eq:11}
  \cN_{n,d} = \left\{ x \in \cN_d \suchthat x^n = 0 \right\}
\end{align}
Notice that $\cN_{n,d} = \overline{\OO_{(n,n,\ldots,n,r)}}$ where $d = k n + r$ for $0 \leq r \leq n-1$ and $n$ occurs $k$ times in $(n,n,\ldots,n,r)$.
Define:
\begin{align}
  \label{eq:7}
  \tcN_{n,d} = \left\{ (x,F_\bullet) \in \cN_{n,d} \times \cF_{n,d} \suchthat  x F_i \subseteq F_{i-1} \text{ for } i \in [n] \right\}
\end{align}
Corresponding to \eqref{eq:12}, we have a connected component decomposition:
\begin{align}
  \label{eq:13}
  \tcN_{n,d} = \bigsqcup_{\dd \in \cP_{n,d}} \tcN_{\dd}
\end{align}
For each $\dd \in \cP_{n,d}$, we have an identification:
\begin{align}
  \label{eq:14}
  T^* \cF_{\dd} = \tcN_{\dd}
\end{align}
Under \eqref{eq:10}, we also can identify
\begin{align}
  \label{eq:18}
  \tcN_{\dd} = G \times^{P_{\dd}} \fu_\dd
\end{align}
where the notation $G \times^{P_{\dd}} \fu_\dd$ denotes quotient of $G \times \fu_\dd$ under the diagonal action of $P_{\dd}$ (acting on the right for the $G$-factor).

For each $\dd \in \cP_{n,d}$, we have the natural map $\mu_{\dd} : \tcN_{\dd} \rightarrow \tcN_{n,d}$. The map $\mu_{\dd}$ is a resolution of singularities. Taking disjoint union, we obtain a map $\mu_{n,d} : \tcN_{n,d} \rightarrow \cN_{n,d}$.
It is known that $\mu_{n,d}$ is \emph{semi-small}.
Notice that $\cN_{d,d} = \cN_d$ and that the map
\begin{align}
  \label{eq:15}
  \mu_{1^d} : \tcN_{1^d} \rightarrow \cN_{d}
\end{align}
is usual Springer resolution.

Fix $\dd \in \cP_{n,d}$. For $x  \in \OO_\lambda \subseteq \cN_{n,d}$, we define $\tcN_{\dd}^x = \mu_{\dd}^{-1}(x)$. We call $\tcN_{\dd}^x$ a \terminology{little Spaltenstein fiber}. It is known that $\dim \tcN_{\dd}^x = \dim \cF_{\dd} - \dim \OO_\lambda$, and that $\tcN_{\dd}^x$ is equidimensional.

\subsubsection{Big Spaltenstein varieties}
\label{sec:big-spaltenstein-varieties}
For each $\dd \in \cP_{n,d}$, define:
\begin{align}
  \label{eq:16}
  \tfg_{n,d} = \left\{ (x,F_\bullet) \in \fg_{d} \times \cF_{n,d} \suchthat  x F_i \subseteq F_{i} \text{ for } i \in [n] \right\}
\end{align}
Notice the difference from \eqref{eq:7}. Corresponding to \eqref{eq:12}, we have:
\begin{align}
  \label{eq:17}
  \tfg_{n,d} = \bigsqcup_{\dd \in \cP_{n,d}} \tfg_{\dd}
\end{align}
For $\dd \in \cP_{n,d}$, we have
\begin{align}
  \label{eq:19}
  \tfg_{\dd} = G\times^{P_{\dd}} \fp_{\dd}
\end{align}
from which we see that $\tfg_{\dd}$ is connected and non-singular with $\dim \tfg_{\dd} = \dim \fg_d = d^2$.

For each $\dd \in \cP_{n,d}$, we have the natural map $p_{\dd} : \tfg_{\dd} \rightarrow \tfg_{d}$.  Taking disjoint union, we obtain a map $p_{n,d} : \tfg_{n,d} \rightarrow \fg_d $. It is known that $p_{n,d}$ is \emph{small}. Notice that the map
\begin{align}
  \label{eq:15}
  \mu_{1^d} : \tfg_{1^d} \rightarrow \fg_{d}
\end{align}
is usual Grothendieck-Springer alteration.

Fix $\dd \in \cP_{n,d}$. For $x  \in \OO_\lambda \subseteq \cN_{d} \subseteq \fg_{d}$, we define $\tfg_{\dd}^x = p_{\dd}^{-1}(x)$. We call $\tfg_{\dd}^x$ a \terminology{big Spaltenstein variety}. It is known that $\dim \tfg_{\dd}^x \leq \dim \cF_{\dd} - \frac{1}{2}\dim \OO_\lambda$. Unlike little Spaltenstein varieties, big Spaltenstein varieties need not be equidimensional.

Observe that for $x \in \cN_{d}$ that the natural inclusion $\tcN^x_{1^d} \hookrightarrow \tfg^x_{1^d}$ induces an isomorphism at the level reduced schemes; in particular, they are homeomorphic in the analytic topology. Because we are only concerned with topological notions, we will consider all the spaces above with their induced reduced scheme structure. In particular, for us $\tcN^x_{1^d} = \tfg^x_{1^d}$. In this case, little and big Spaltenstein varieties coincide, and we will call them \terminology{Springer fibers} following the usual terminology.

\subsubsection{Springer's Weyl group action}
Recall that there is an $S_d$-action on $H_\bullet(\cF_{1^d})$. There are two natural choices for such an action, which differ by tensoring with the sign representation. We fix the choice where $H_0(\cF_{1^d})$ is the trivial $S_d$ module. 

Fix $\lambda \vdash d$ and $x \in \OO_\lambda \subseteq \cN_{d}$. Let $d_\lambda = \dim \cF_{1^d} - \frac{1}{2} \dim \OO_\lambda$. Then we have that the Springer fiber $\tcN^x_{1^d}$ has dimension $d_\lambda$. 

We have a natural closed embedding $\tcN^x_{1^d} \hookrightarrow \cF_{1^d}$ that induces a map 
\begin{align}
  \label{eq:20}
  H_{2 d_{\lambda}}(\tcN^x_{1^d}) \hookrightarrow H_{\bullet}(\cF_{1^d})
\end{align}
that is an injection \cite[Theorem 6.5.2]{Chriss-Ginzburg} (in fact, the entire homology $H_{\bullet}(\tcN^x_{1^d})$ embeds into $H_{\bullet}(\cF_{1^d})$ in type A, but we will not need this harder fact). The subspace $H_{2 d_{\lambda}}(\tcN^x_{1^d})$ is preserved under the $S_d$ -action, and we have an isomorphism $H_{2 d_{\lambda}}(\tcN^x_{1^d}) \cong S(\lambda)$ of $S_d$-modules. As explained in {\it loc. cit.}, this is exactly the action coming from Springer theory. We will call this the \terminology{Springer action} on $H_{2 d_{\lambda}}(\tcN^x_{1^d})$.

\subsection{Convolution in equivariant Borel-Moore homology}

The main method that we will use to study representations is that of convolution in Borel-Moore homology as developed in \cite[Chapter 2]{Chriss-Ginzburg}. The method developed there does not explicitly involve equivariant Borel-Moore homology, but as usual equivariant Borel-Moore homology is constructed using finite-dimension approximations of the classifying space (see e.g. \cite{Lusztig-1988,Edidin-Graham}), so the constructions from the non-equivariant situation carry over.

We will not recall the full details of convolution construction but will recall briefly how to compute using convolution in $A$-equivariant Borel-Moore homology where $A$ be a linear algebraic torus, i.e. a finite product of $\GG_m$'s. In our applications, we will have $A = T \times \GG_m$.  Our primary computational tool will be the localization theorem. We note that we will only apply the localization theorem to algebraic cycles, so we will make use of Brion's \cite{Brion} formulas for localization in equivariant Chow groups, which upon applying cycle-class map, will give us the corresponding formulas in equivariant Borel-Moore homology.

\subsubsection{Euler classes and equivariant multiplicities}

Let $Z$ be a variety with an $A$-action so that $Z^A$ is finite and that all fixed points $z \in Z^A$ are non-degenerate, which is to say that the zero weight does not appear in the tangent space $T_zZ$. This non-degeneracy condition is immediately verified if we can $A$-equivariantly embed $Z$ into a non-singular variety with finitely many $A$-fixed points; this will be the case for all varieties we consider.

If $z \in Z^A$ is a non-singular point of of $Z$, then we will define the $A$-equivariant Euler class of $Z$ at $z$, denoted $\eu_z(Z)$, to be the product of the $A$-weights of the tangent space $T_z(Z)$ counted with multiplicity. A special case of this is an $A$-representation $M$ with $M^A = 0$. In this case, the Euler class $\eu_0(M)$ is simply the product of the weights of $M$ counted with multiplicity. In this case, we will drop the subscript $0$ and simply write $\eu(M)$ for the Euler class of $M$ at $0$.

If $z \in Z^A$ is a possibly singular point of $Z$, then there is a natural generalization of the Euler class called the \terminology{$A$-equivariant multiplicity of $z$ at $Z$} denoted $e_z^A(Z)$. We refer to \cite[\S 4.2]{Brion} for the precise definition. This was also considered earlier by Joseph \cite{Joseph} and Rossmann \cite{Rossmann}. The notion is closely related to the notion of multidegree in commutative algebra (see e.g. \cite[]{Miller-Sturmfels}); in particular, it is computable by the methods of computational commutative algebra. The relevant facts for us are the localization theorem (Theorem \ref{thm:brions-localization-thm} below), which in fact characterizes the equivariant multiplicities, and the fact that if $z \in Z^A$ is a non-singular point of $Z$ then $e_z^A(Z) = \frac{1}{\eu_z(Z)}$.

\subsubsection{Equivariant multiplicities and the localization theorem}
\label{sec:localization-theorem}
Let $S = H^\bullet_T(\pt)$ denote the $A$-equivariant cohomology of a point, and let $Q$ denote the fraction field of $S$. Let $Z$ be an $A$-variety as above with finitely many $A$-fixed points, all of which are non-degenerate.

Let $i: Z^A \hookrightarrow Z$ denote the closed embedding of the fixed points. We have a proper pushforward map
\begin{align}
  \label{eq:26}
   H^A_\bullet(Z^A) \rightarrow H^A_\bullet(Z)
\end{align}
that is an $S$-module map. Explicitly, $H^A_\bullet(Z^A)$ is a free $S$-module spanned by the equivariant fundamental classes $[z]$ of fixed points $z \in Z^A$. We can tensor up with $Q$ to obtain a map:
\begin{align}
  \label{eq:27}
   Q \otimes_S H^A_\bullet(Z^A) \rightarrow Q \otimes_S H^A_\bullet(Z)
\end{align}

Let $Y \subseteq Z$ be a $A$-invariant closed subvariety. Then we can consider the $A$-equivariant fundamental class $[Y] \in H^A_\bullet(Z)$. We can also consider $[Y] \in Q \otimes_S H^A_\bullet(Z)$. Then we have the following formula (see e.g. \cite[Corollary 4.2]{Brion}), which is the version of the localization theorem that is necessary for our purposes.

\begin{Theorem}
  \label{thm:brions-localization-thm}
  As classes in $Q \otimes_S H^A_\bullet(Z)$, we have:
  \begin{align}
    \label{eq:28}
  [Y] = \sum_{z \in Z^A} e_z(Y) \cdot [z]
  \end{align}
\end{Theorem}

\subsubsection{Computing convolution at fixed points}
Convolution involves three operations: smooth pullback, refined intersection, and proper pushforward. We will recall how these operations behave for fundamental classes of fixed points. Using \eqref{eq:28}, we will be able to calculate for more general $A$-invariant fundamental classes.

Proper pushforward is easy to understand: the pushforward of a fixed point is a fixed point. Smooth pullback is also fairly straightforward: the smooth pullback of a fixed point is the fundamental class of the fiber, which is non-singular by assumption. Furthermore, these operations commute with the action of $S$.

The most subtle part of convolution is the operation of refined intersection. Let $X$ be a non-singular $A$-variety; for simplicity, let us suppose $X$ has finitely many $A$-fixed points.  Let $Z_1,Z_2 \subseteq X$ be $A$-invariant closed subvarieties, and let $Z_1 \cap Z_2$ be their intersection. Let $m = \dim X$, and $i,j \in \ZZ$. Then we can consider the refined intersection pairing 
\begin{align}
  \label{eq:25}
\cap : H_i(Z_1) \otimes H_j(Z_2) \rightarrow H_{i+j-2m}(Z_1 \cap Z_2)
\end{align}
relative to the ambient space $X$. We can also consider this $A$-equivariantly and at fixed points. The result is the following commutative diagram:
\begin{equation}
  \label{eq:29}
  \begin{tikzcd}
   H^A_i(Z_1^A) \otimes H^A_j(Z_2^A) \arrow[r] \arrow[d,"\cap"] & H^A_i(Z_1) \otimes H^A_j(Z_2) \arrow[r] \arrow[d,"\cap"] & H_i(Z_1) \otimes H_j(Z_2) \arrow[d,"\cap"]  & \\
   H^A_{i+j-2m}(Z_1^A \cap Z_2^A)     \arrow[r] & H^A_{i+j-2m}(Z_1 \cap Z_2)     \arrow[r] &  H_{i+j-2m}(Z_1
   \cap Z_2)    
  \end{tikzcd}
\end{equation}
All the maps in the left square commute with the action of $H^\bullet_A(\pt)$, so to calculate $H^A_i(Z_1^A) \otimes H^A_j(Z_2^A) \overset{\cap}{\rightarrow} H^A_{i+j-2m}(Z_1^A \cap Z_2^A)$ it suffices to calculate the intersection
$[z_1] \cap [z_2]$ for $z_1 \in Z_1^A$ and $z_2 \in Z_2^A$. This intersection is given by the excess intersection formula (see e.g \cite[Proposition 2.6.47]{Chriss-Ginzburg} and \cite[Corollary 6.3]{Fulton}):
\begin{align}
  \label{eq:31}
  [z_1] \cap [z_2] =
  \begin{cases}
    \eu_z(X) \cdot [z] &\textif{ z = z_1 = z_2} \\
    0 &\textif{z_1 \neq z_2}
  \end{cases}
\end{align}

\section{Some geometric $\sl_n$-representations}

We will recall three geometric constructions of $\sl_n$-modules: Ginzburg's Lagrangian construction, Mirkovi\'c and Vilonen's Geometric Satake equivalence, and Braverman and Gaitsgory's construction via Schur-Weyl duality and Springer theory. All three constructions give rise to geometrically defined bases. Savage \cite{Savage} has shown that Ginzburg's basis agrees with the basis arising in the quiver variety construction of Nakajima \cite{Nakajima}. Braverman, Gaitsgory, and Vybornov have shown that the Braverman-Gaitsgory basis coincides with the Mirkovi\'c-Vilonen basis in type A.

\subsection{Ginzburg's construction of $\sl_n$-representations}

We recall Ginzburg's Lagrangian construction of $\sl_n$-representation (\cite{Ginzburg-1991} and \cite[Chapter 4]{Chriss-Ginzburg}). For any pair $\dd',\dd \in \cP_{n,d}$, we can form the ``Steinberg'' variety $\tcN_{\dd'} \times_{\cN_{n,d}} \tcN_{\dd}$, which is equal to the union of conormal bundles to the (diagonal) $G$-orbits on $\cF_{\dd'} \times \cF_{\dd}$. Let $Y_{\dd',\dd}$ denote the unique minimal (and hence closed) $G$-orbit on $\cF_{\dd'} \times \cF_{\dd}$, and let $Z_{\dd',\dd}$ denote the total space of the conormal bundle to $Y_{\dd',\dd}$ inside
$\cF_{\dd'} \times \cF_{\dd}$.

Similarly to \eqref{eq:10} and~\eqref{eq:18}, we have
\begin{align}
  \label{eq:30}
 Y_{\dd',\dd} = G/(P_{\dd'} \cap P_{\dd}) 
\end{align}
and:
\begin{align}
  \label{eq:32}
 Z_{\dd',\dd} = G \times^{P_{\dd'} \cap P_{\dd}} (\fu_{\dd'} + \fu_\dd)
\end{align}

For any $x \in \cN_{n,d}$, we have an action by convolution:
\begin{align}
  \label{eq:33}
[Z_{\dd',\dd}] \star - : H_\bullet(\tcN^x_{\dd}) \rightarrow H_{\bullet + 2\dim\tcN^x_{\dd'} - 2\dim\tcN^x_{\dd}}(\tcN^x_{\dd'})
\end{align}
Notice that $Z_{\dd',\dd}$ is invariant under $T \times \GG_m$, where the $\GG_m$ factor acts by scaling cotangent fibers. Therefore, we can consider $[Z_{\dd',\dd}] \in H^{T \times \GG_m}_\bullet(\tcN_{\dd'} \times_{\cN_{n,d}} \tcN_{\dd})$.
In the special case of $x = 0$, $\tcN^0_{\dd} = \cF_\dd$ and $\tcN^x_{\dd'} = \cF_{\dd'}$ are $T \times \GG_m$ invariant, so we can consider the convolution action $T \times \GG_m$-equivariantly
\begin{align}
  \label{eq:34}
[Z_{\dd',\dd}] \star - : H^{T \times \GG_m}_\bullet(\cF_{\dd}) \rightarrow H^{T \times \GG_m}_{\bullet + 2\dim\cF_{\dd'} - 2\dim\cF_{\dd}}(\cF_{\dd'})
\end{align}
which recovers \eqref{eq:33} by setting $T \times \GG_m$-equivariant parameters equal to $0$.

\subsubsection{Chevalley generators}
Given $\dd = (d_1,\ldots,d_n) \in \cP_{n,d}$ and $a \in [n-1]$, we can form
\begin{align}
  \label{eq:60}
  \e_a(\dd) =
  \begin{cases}
    (d_1,\ldots,d_{a-1},d_a+1,d_{a+1}-1,d_{a+2},\ldots,d_n) &\textif d_{a+1} \geq 1 \\
    \nabla &\text{ otherwise }
  \end{cases}
\end{align}
where $\nabla$ is the ``ghost'' composition.
Similarly, we form:
\begin{align}
  \f_a(\dd) =
  \begin{cases}
    (d_1,\ldots,d_{a-1},d_a-1,d_{a+1}+1,d_{a+2},\ldots,d_n) &\textif d_{a} \geq 1 \\
    \nabla &\text{ otherwise }
  \end{cases}
\end{align}

For each $a \in [n-1]$, define
\begin{align}
  \label{eq:35}
  E_a = \sum_{\dd \in \cP_{n,d} : \e_a \dd \neq \nabla} [Z_{\e_a \dd,\dd}] 
\end{align}
and
\begin{align}
  \label{eq:36}
  F_a = \sum_{\dd \in \cP_{n,d} : \e_a \dd \neq \nabla} (-1)^{\sgn_a(\dd)} \left[ T^*_{Y_{\dd^-_a,\dd}}  \left(\cF_{\dd^-_a}\times\cF_{\dd}\right)\right]
\end{align}
where for $\dd = (d_1,\ldots,d_k)$, $\sgn_a(\dd) = d_{a+1} - d_a + 1$.

\begin{Remark}
We refer to \cite[Remark 3.6]{Savage} for an explanation of the sign in \eqref{eq:36}. There are many other possible sign choices that would work. For example, Vasserot \cite{Vasserot-1993} picks signs in a way that correspond to orienting the fundamental class of cotangent bundles via their symplectic structure. This does not necessarily agree with the orientations coming from their complex structure, but it leads to nicer formulas.
\end{Remark}

For any $x \in \cN_{n,d}$, we consider $E_a$ and $F_a$ operating on $\bigoplus_{\dd \in \cP_{n,d}} H_\bullet(\tcN^x_\dd)$ by convolution; note that these operators do not preserve the grading. However, these operators do preserve the top-degree homology, and we have the following theorem of Ginzburg.

\begin{Theorem}[\cite{Ginzburg-1991,Chriss-Ginzburg}]
  \label{thm:ginzburg-lagrangian-construction}
Let $\lambda \in \cP^{++}_{n,d}$, and let $x \in \OO_{\lambda^t} \subseteq \cN_{n,d}$. The operators $\left\{E_a,F_a \right\}_{a \in [n-1]}$ define an $\sl_n$-action on 
$\bigoplus_{\dd \in \cP_{n,d}} H_\bullet(\tcN^x_\dd)$. As an $\sl_n$-module, $\bigoplus_{\dd \in \cP_{n,d}} H_\top(\tcN^x_\dd)$
is irreducible and is isomorphic to $V(\lambda^t)$.
\end{Theorem}

\subsection{Geometric Satake and the Mirkovi\'c-Vilonen basis}
\newcommand{\Sat}{\mathrm{Sat}}
\newcommand{\Rep}{\mathrm{Rep}}
\newcommand{\Vect}{\mathrm{Vect}}

Write $\cO = \CC[[t]]$  and $\cK = \CC((t))$. Let $\bG$ be a complex reductive algebraic group with a fixed pair of opposite Borels $\bB$ and $\bB^-$. Let $\bU$ and $\bU^-$ be the unipotent radicals of the Borels, and let $\bT = \bB \cap \bB^-$ be the corresponding maximal torus.  We form the {\it affine Grassmannian} $\Gr  = \Gr_{\bG} = \bG(\cK)/\bG(\cO)$, which has the structure of an projective ind-scheme of ind-finite-type. 

For each coweight $\mu$ of the torus $\bT$, there is a corresponding element $\mu \in \Gr$. The $\bU(\cK)$-orbits on $\Gr$ are indexed by coweights and are precisely the sets $S_\mu = \bU(\cK) \mu$ for coweights $\mu$. The $\bG(\cO)$-orbits on $\Gr$ are indexed by dominant coweights $\lambda$ and are precisely the sets  $\Gr^\lambda = \bG(\cO) \lambda$ for dominant coweights $\lambda$.

\subsubsection{The Geometric Satake action and Mirkovi\'c-Vilonen cycles}

One considers the {\it Satake category} $\Sat$ consisting of $\bG(\cO)$-equivariant perverse sheaves (we will only consider sheaves with complex coefficients) on $\Gr_{\bG}$. The category $\Sat$ has a symmetric monoidal structure given by convolution of sheaves.

Let $\bG^\vee$ denote the Langlands dual group of $\bG$, and let $\Rep(\bG^\vee)$ denote the category of finite-dimensional $\bG^\vee$-representations.
The celebrated Geometric Satake equivalence is the following explicit description of the Satake category.
\begin{Theorem}[\cite{Ginzburg-1995,Mirkovic-Vilonen}]
  \label{thm:geometric-satake}
  There is an equivalence of symmetric monoidal categories
  \begin{align}
    \label{eq:104}
   \Sat \overset{\sim}{\rightarrow} \Rep(\bG^\vee) 
  \end{align}
  such that the forgetful functor $\Rep(\bG^\vee) \rightarrow \Vect$ corresponds to the hypercohomology functor $\HH^\bullet : \Sat \rightarrow \Vect$.
\end{Theorem}

The Mirkovi\'c-Vilonen proof of Geometric Satake provides finer information by geometrically realizing the decomposition of $\bG^\vee$-representations into weight spaces. Let $\lambda$ be a dominant coweight and let $\mu$ be any coweight. The relevant fact for us is that there is a canonical isomorphism:
\begin{align}
  \label{eq:106}
   V(\lambda)_\mu = H^{2 \cdot \height(\lambda-\mu)}_{c}(\overline{\Gr^\lambda} \cap S_\mu)
\end{align}
Here $\height(\lambda-\mu)$ denotes the height of the coweight $\lambda-\mu$, and $H^{2 \cdot\height(\lambda-\mu)}_{c}$ denotes compactly supported cohomology. Furthermore, Mirkovi\'c and Vilonen prove that $\overline{\Gr^\lambda} \cap S_\mu$ is equidimensional of dimension $\height(\lambda-\mu)$. The elements of $\Irr(\overline{\Gr^\lambda} \cap S_\mu)$ are called {\it Mirkovi\'c-Vilonen (MV) cycles (of weight $(\lambda,\mu)$)}. We see therefore that the weight space $V(\lambda)_\mu$ has a basis (the {\it Mirkovi\'c-Vilonen (MV) basis}) indexed by the set of MV cycles of weight $(\lambda,\mu)$.

\subsubsection{Specifics in type A}

We will be interested in the case of $\bG = \GL_n$ and $\bG = \SL_n$. The group $\GL_n$ is its own Langlands dual group, and there is a natural bijection between weights and coweights: both sets are naturally in bijection with $\ZZ^n$. For $\SL_n$, the Langlands dual group is $\mathrm{PGL}_n$, and there are more weights than coweights. We can identify the set of coweights with a subset of weights: the coweights are identified precisely with the root lattice. We will use these identifications freely.

The affine Grassmannian $\Gr_{\GL_n}$ is a disjoint union of connected components $\Gr^d_{\GL_n}$.
A partition $\lambda \in \cP^{++}_{n,d}$ gives rise to a dominant coweight of $\lambda$ of $\GL_n$. The corresponding point $\lambda$ of the affine Grassmannian lies in $\Gr^d_{\GL_n}$.
Multiplication by the scalar matrix $t$ induces isomorphisms $\Gr^d_{\GL_n} \overset{\sim}{\rightarrow} \Gr^{d+n}_{\GL_n}$. The map $t$ induces an auto-equivalence of the Satake category. Under Geometric Satake, this auto-equivalence corresponds precisely to tensoring with the determinant character of the Langlands dual group. Because the determinant is trivial on $\SL_n$ (viewed as a subgroup of the Langlands dual group), the auto-equivalence commutes with the $\SL_n$-action coming from Geometric Satake. It also commutes with the Mirkovi\'c-Vilonen construction of weight spaces for the maximal torus of $\SL_n$ (but not for $\GL_n$).

The affine Grassmannian $\Gr_{\SL_n}$ is connected, and the inclusion $\SL_n \hookrightarrow \GL_n$ induces a homeomorphism between $\Gr_{\SL_n}$ and $\Gr^0_{\GL_n}$. So we can realize the Satake category for $\SL_n$ as a subcategory of the $\GL_n$ Satake category.

We will be primarily interested in the case when $n=d$, i.e. we consider partitions $\lambda \in \cP^{++}_{d,d}$. In this case, we will shift the emphasis of our notation slightly and write $\SL_d$ instead of $\SL_n$ (of course $n=d$ in this case). Recall that we can consider partitions in $\cP^{++}_{d,d}$ as weights for $\SL_d$. In this case the weights lie in the root lattice and are also coweights. The weights that appear are precisely the weights that are less than or equal to $d \omega_1$ in the dominance order. However, by our above discussion, such partitions naturally correspond to points in the $d$-th connected component $\Gr^d_{\GL_d}$. Multiplication by $t$ identifies $\Gr^0_{\GL_d}$ with $\Gr^d_{\GL_d}$. So although we will primarily interested in the representations and MV basis arising in the $\SL_d$ Satake category, we will sometimes equivalently work in $\Gr^0_{\GL_d}$ and $\Gr^d_{\GL_d}$.

\subsection{The Braverman-Gaitsgory action}

Braverman, Gaitsgory, and Vybornov \cite{Braverman-Gaitsgory-Vybornov} have shown that the Mirkovi\'c-Vilonen basis in type A can be realized via an action of $\sl_n$ on the cohomology of big Spaltenstein varieties. This action was first constructed by Braverman and Gaitsgory in \cite{Braverman-Gaitsgory}.

\subsubsection{The construction}

Let $\E$ be as in section \ref{sec:rep-theory}, i.e. $\E$ is an $n$-dimensional vector space equipped with a decomposition $\E = \E_1 \oplus \cdots \oplus \E_n$ into lines. Recall that this decomposition specifies a maximal torus $D$ in $\GL(\E)$.

Recall the map $p_{n,d} : \tfg_{n,d} \rightarrow \fg_d$. Following Braverman and Gaitsgory \cite{Braverman-Gaitsgory}, one defines a sheaf $\cK^{n,d}_\E$ on $\tfg_{n,d}$ by declaring:
\begin{align}
 \label{eq:37}
\cK^{n,d}_\E \Big\vert_{\tfg_{\dd}} = \E^\dd[d^2]
\end{align}
Note that $\cK^{n,d}_\E$ carries an action of $D$ by construction.
The degree shift ensures that $\cK^{n,d}_\E$ is a perverse sheaf.
We also consider the usual Grothendieck-Springer alteration $p_{1^d} : \tfg_{1^d} \rightarrow \tfg_d$, and we define the 
Grothendieck sheaf:
\begin{align}
  \label{eq:38}
  \Gro = (p_{1^d})_* \CC_{\tfg_{1^d}}[d^2]
\end{align}
The Grothendieck sheaf carries an action of $S_d$, so it makes sense to form $ (\Gro \otimes \E^{\otimes d})^{S_d}$, which by construction will carry a $\GL(\E)$-action. Braverman and Gaitsgory prove the following.
\begin{Theorem}[{\cite[\S 2.6]{Braverman-Gaitsgory}}]
  \label{thm:brav-gaits-t-equiv-isom}
  There is a natural $D$-equivariant isomorphism:
  \begin{align}
    \label{eq:40}
   (\Gro \otimes \E^{\otimes d})^{S_d} \overset{\sim}{\rightarrow}  (p_{n,d})_*\cK^{n,d}_\E
  \end{align}
\end{Theorem}
One can therefore transport the $\GL(\E)$-action to $(p_{n,d})_*\cK^{n,d}_\E$. Let $\lambda \vdash d$, let $x \in \OO_\lambda \subset \cN_{d} \subset \fg_d$, and let $i_x : \{x\} \hookrightarrow \fg_d$ denote the inclusion of $x$. Using the basis of $\E$, we can identify the $!$-stalk $i_x^! (p_{n,d})_*\cK^{n,d}_\E$ with the Borel-Moore homology $\bigoplus_{\dd \in \cP_{n,d}} H_\bullet(\tfg^x_\dd)$. By \eqref{eq:40} we have a $\GL(\E)$-action on $\bigoplus_{\dd \in \cP_{n,d}} H_\bullet(\tfg^x_\dd)$. Because $\tfg_{n,d}$ is equidimensional, this action preserves homological degrees. In particular,
\begin{align}
  \label{eq:39}
 \bigoplus_{\dd \in \cP_{n,d}} H_{2 d_\lambda}(\tfg^x_\dd) 
\end{align}
caries a $\GL(\E)$-action. Because $\dim \tfg^x_\dd \leq d_\lambda$, $H_{2 d_\lambda}(\tfg^x_\dd)$ has a basis indexed by $\Irr_{d_\lambda}(\tfg^x_\dd)$.

Similarly because $p_{n,d}$ is a proper map, by taking $*$-stalks we have $\GL(\E)$-action on $\bigoplus_{\dd \in \cP_{n,d}} H^\bullet(\tfg^x_\dd)$. Again cohomological degree is preserved, and we have a $\GL(\E)$-action on:
\begin{align}
  \label{eq:41}
 \bigoplus_{\dd \in \cP_{n,d}} H^{2 d_\lambda}(\tfg^x_\dd) 
\end{align}
If $\lambda \in \cP^{++}_{n,d}$, this module is isomorphic to $V(\lambda)$. Otherwise, this module is $0$ (see e.g. \cite[Corollary 3.5(a)]{Borho-MacPherson}). 
Again, $H^{2 d_\lambda}(\tfg^x_\dd)$ has a basis indexed by $\Irr_{d_\lambda}(\tfg^x_\dd)$. Braverman, Vybornov, and Gaitsgory  have proved the following.

\begin{Theorem}{\cite{Braverman-Gaitsgory-Vybornov}}
  \label{thm:bgv-comparison-with-mv-basis}
 Under the $\GL(\E)$-isomorphism  
 \begin{align}
   \label{eq:42}
 V(\lambda) = \bigoplus_{\dd \in \cP_{n,d}} H^{2 d_\lambda}(\tfg^x_\dd) 
 \end{align}
 the basis $\bigsqcup_{\dd \in \cP_{n,d}} \Irr_{d_\lambda}(\tfg^x_\dd)$ coincides with the Mirkovi\'c-Vilonen basis.
\end{Theorem}

For our calculations, we will focus on the $\GL(\E)$-action \eqref{eq:39} on Borel-Moore homology. In particular, we will show that the action of Chevalley generators is given by convolution by explicit correspondences. The action on cohomology will then be obtained via a straightforward duality procedure (see Proposition \ref{prop:dual-chevalley-generators} below).

\section{Ginzburg's action and weight-zero spaces for small representations}
\label{sec:ginz-weight-zero-small}

We will consider the case of $n=d$ for the remainder of this section. Fix $\lambda \in \cP^{++}_{d,d}$, i.e., $\lambda$ is a partition of $d$. Let $x \in \OO_{\lambda^t} \subset \cN_{d}$. Recall Ginzburg's $\sl_n$-action on
\begin{align}
  \label{eq:47}
  \bigoplus_{\dd \in \cP_{d,d}} H_\top(\tcN^x_\dd)
\end{align}
that realizes the irreducible module $V(\lambda)$. Recall the composition $1^d = (1,1,\ldots,1) \in \cP_{d,d}$. The space $H_\top(\tcN^x_{1^d})$ is precisely the zero weight space of $V(\lambda)$ as an $\sl_n$-module (warning: this is not true for the usual $\gl_n$-module structure where the center acts non-trivially). Therefore, the Weyl group $S_d$ acts on $H_\top(\tcN^x_{1^d})$.

Notice however, that $H_\top(\tcN^x_{1^d})$ is exactly the top homology of a Springer fiber. Therefore, the Weyl group $S_d$ also acts $H_\top(\tcN^x_{1^d})$ via Springer theory. In this section we will compare these actions and show that they are equal.

\begin{Remark}
  It is well known that Springer's Weyl group action $H_\top(\tcN^x_{1^d})$ and the Weyl group action on the zero weight space of $V(\lambda)$ both realize the Specht module $S(\lambda^t)$ (e.g. \cite{Gutkin} and \cite{Kostant}). So by comparing the actions, we are essentially comparing the two natural bases. 
\end{Remark}

\subsection{Action on zero weight space}
\label{subsec:ginz-action-on-zero-weight-space}

For each $a \in [d-1]$, recall that 
\begin{align}
  \label{eq:49}
  T_a \equiv 1 - E_a F_a 
\end{align}
as operators on $H_\top(\tcN^x_{1^d})$. 

\subsubsection{Computing the action of $T_a$ via convolution}
\label{subsubsect-computing-T-a}

We need to compute the operator:
\begin{align}
  \label{eq:52}
E_a F_a : H_\top(\tcN^x_{1^d}) \rightarrow H_\top(\tcN^x_{1^d}) 
\end{align}
The operator
\begin{align}
  \label{eq:50}
  F_a : H_\top(\tcN^x_{1^d}) \rightarrow H_\top(\tcN^x_{\f_a 1^d})
\end{align}
is given by convolution by:
\begin{align}
  \label{eq:51}
 -[Z_{\f_a 1^d,1^d}] = \sum \sum_{x \in S_{d} / (S_{\f_a 1^d} \cap S_{1^d})} x \left(\frac{1}{ \eu_0(\fu^-_{\f_a 1^d}+ \fu^-_{1^d})\eu_0( (\fu_{\f_a 1^d} \cap \fu_{1^d}) \otimes \CC_h )} \right) [xS_{\f_a 1^d},xS_{1^d}] 
\end{align}
Similarly, 
\begin{align}
  \label{eq:53}
  E_a : H_\top(\tcN^x_{\f_a1^d}) \rightarrow H_\top(\tcN^x_{1^d})
\end{align}
is given by convolution by:
\begin{align}
  \label{eq:54}
 [Z_{1^d,\f_a1^d}] = \sum \sum_{x \in S_{d} / (S_{1^d} \cap S_{\f_a 1^d})} x \left(\frac{1}{ \eu_0(\fu^-_{1^d}+ \fu^-_{\f_a 1^d})\eu_0( (\fu_{1^d} \cap \fu_{\f_a 1^d}) \otimes \CC_h )} \right) [xS_{1^d},xS_{\f_a 1^d}] 
\end{align}

Consider the triple product
\begin{align}
  \label{eq:56}
  \tcN_{1^d} \times \tcN_{\f_a 1^d}\times \tcN_{1^d}
\end{align}
We have projection maps
\begin{align}
  \label{eq:57}
p_{12} : \tcN_{1^d} \times \tcN_{\f_a 1^d}\times \tcN_{1^d} \rightarrow  \tcN_{1^d} \times \tcN_{\f_a 1^d}   \\
  p_{23} : \tcN_{1^d} \times \tcN_{\f_a 1^d}\times \tcN_{1^d} \rightarrow  \tcN_{\f_a 1^d} \times \tcN_{1^d}   \\
  p_{13} : \tcN_{1^d} \times \tcN_{\f_a 1^d}\times \tcN_{1^d} \rightarrow  \tcN_{1^d} \times \tcN_{1^d}   
\end{align}
Let us compute the convolution:
\begin{align}
  \label{eq:55}
[Z_{1^d,\f_a1^d}] \star [Z_{\f_a 1^d,1^d}]  =  (p_{13})_* \left(p_{12}^*( [Z_{1^d,\f_a 1^d}]) \cap   p_{23}^*( [X_{\f_a 1^d,1^d}]) \right)
\end{align}

First compute
\begin{align}
  \label{eq:58}
  &p_{23}^*( [Z_{\f_a 1^d,1^d}]) = \sum_{w \in S_d} \sum_{x \in S_d/(S_{\f_a1^d} \cap S_{1^d})}  \\ & w\left(\frac{1}{\eu_{0}(\fu^-_{1^d}) \eu_0(\fu_{1^d} \otimes \CC_h)}\right) x \left(\frac{1}{ \eu_0(\fu^-_{\f_a 1^d}+ \fu^-_{1^d})\eu_0( (\fu_{\f_a 1^d} \cap \fu_{1^d}) \otimes \CC_h )} \right) [wS_{1^d},xS_{\e_a 1^d},xS_{1^d}] =\\
 & \sum_{w \in S_d} \sum_{x \in S_d}   w\left(\frac{1}{\eu_{0}(\fu^-_{1^d}) \eu_0(\fu_{1^d} \otimes \CC_h)}\right) x \left(\frac{1}{ \eu_0(\fu^-_{ 1^d})\eu_0( \fu_{\f_a 1^d} \otimes \CC_h )} \right) [wS_{1^d},xS_{\e_a 1^d},xS_{1^d}] 
\end{align}
as $  \fu^-_{ 1^d} = \fu^-_{\e_a 1^d}+ \fu^-_{1^d} $, $ \fu_{\f_a1^d} = \fu_{\f_a 1^d} \cap \fu_{1^d} $, and  $\# S_{\f_a 1^d} \cap S_{1^d} = 1$.
Similarly,
\begin{align}
  \label{eq:59}
& p_{12}^*( [Z_{1^d,\f_a 1^d}]) = \sum_{v \in S_d} \sum_{y \in S_d/(S_{\f_a1^d} \cap S_{1^d})} \\
&  
y \left(\frac{1}{ \eu_0(\fu^-_{\f_a 1^d}+ \fu^-_{1^d})\eu_0( (\fu_{\f_a 1^d} \cap \fu_{1^d}) \otimes \CC_h )} \right)
v\left(\frac{1}{\eu_{0}(\fu^-_{1^d}) \eu_0(\fu_{1^d} \otimes \CC_h)}\right)
[yS_{1^d},yS_{\e_a 1^d},vS_{1^d}] = \\
&  \sum_{v \in S_d} \sum_{y \in S_d}  
y \left(\frac{1}{ \eu_0(\fu^-_{1^d})\eu_0( \fu_{\f_a1^d} \otimes \CC_h )} \right)
v\left(\frac{1}{\eu_{0}(\fu^-_{1^d}) \eu_0(\fu_{1^d} \otimes \CC_h)}\right)
[yS_{1^d},yS_{\e_a 1^d},vS_{1^d}] 
\end{align}

When we compute the refined intersection $p_{12}^*( [Z_{1^d,\f_a 1^d}]) \cap  p_{23}^*( [X_{\f_a 1^d,1^d}])$, we need $w = y$, $x = v$, and $xS_{\f_a 1^d} = yS_{\f_a 1^d}$. Because $S_{\f_a 1^d} = \{1,s_a\}$, we must have $y = xs_a$ or $x = y$. So we have:
\begin{align}
  \label{eq:62}
&p_{12}^*( [Z_{1^d,\f_a 1^d}]) \cap   p_{23}^*( [Z_{\f_a 1^d,1^d}]) = \\   \sum_{x \in S_d} & 
x\left(\frac{\eu_{0}(\fu^-_{1^d}) \eu_0(\fu_{1^d} \otimes \CC_h)}{\eu_{0}(\fu^-_{1^d}) \eu_0(\fu_{1^d} \otimes \CC_h)}\right) x \left(\frac{\eu_0(\fu^-_{\f_a 1^d})\eu_0( \fu_{\f_a 1^d} \otimes \CC_h ) }{ \left( \eu_0(\fu^-_{ 1^d})\eu_0( \fu_{\f_a 1^d} \otimes \CC_h ) \right)^2} \right) \times \\ &x\left(\frac{\eu_{0}(\fu^-_{1^d}) \eu_0(\fu_{1^d} \otimes \CC_h)}{\eu_{0}(\fu^-_{1^d}) \eu_0(\fu_{1^d} \otimes \CC_h)}\right)
                                                                           [xS_{1^d},xS_{\f_a 1^d},xS_{1^d}]
  \\
& + xs_a\left(\frac{\eu_{0}(\fu^-_{1^d}) \eu_0(\fu_{1^d} \otimes \CC_h)}{\eu_{0}(\fu^-_{1^d}) \eu_0(\fu_{1^d} \otimes \CC_h)}\right) xs_a \left(\frac{\eu_0(\fu^-_{\f_a 1^d})\eu_0( \fu_{\f_a 1^d} \otimes \CC_h ) }{  \eu_0(\fu^-_{ 1^d})\eu_0( \fu_{\f_a 1^d} \otimes \CC_h ) } \right) \times
  \\ & x\left(\frac{\eu_{0}(\fu^-_{1^d}) \eu_0(\fu_{1^d} \otimes \CC_h)}{\eu_{0}(\fu^-_{1^d})^2 \eu_0(\fu_{\f_a 1^d} \otimes \CC_h)\eu_0(\fu_{1^d} \otimes \CC_h)}\right)
                                                                           [xs_aS_{1^d},xs_aS_{\f_a 1^d},xS_{1^d}]
\end{align}
Therefore:
\begin{align}
  \label{eq:63}
  & (p_{13})_* \left(p_{12}^*( [Z_{1^d,\f_a 1^d}]) \cap   p_{23}^*( [X_{\f_a 1^d,1^d}]) \right)  = \\
&  \sum_{x \in S_d} 
x \left( \frac{\eu_0(\fu^-_{\f_a 1^d}) }{\eu_0(\fu^-_{ 1^d})} \right)  x \left(\frac{}{  \eu_0(\fu^-_{ 1^d})\eu_0( \fu_{\f_a 1^d} \otimes \CC_h ) } \right) [xS_{1^d},xS_{1^d}]
  \\
& +  xs_a \left(\frac{\eu_0(\fu^-_{\f_a 1^d}) }{  \eu_0(\fu^-_{ 1^d}) } \right)
  x\left(\frac{ 1} {\eu_{0}(\fu^-_{1^d}) \eu_0(\fu_{\f_a 1^d} \otimes \CC_h)}\right)  [xs_aS_{1^d},xS_{1^d}] = \\
& \sum_{x \in S_d} x \left( \frac{\alpha_a+h}{-\alpha_a} \right) 
  x \left(\frac{1}{ \eu_0(\fu^-_{ 1^d})\eu_0( \fu_{1^d} \otimes \CC_h ) } \right) [xS_{1^d},xS_{1^d}] \\ & + \frac{x(\alpha_a+h)}{xs_a(-\alpha_a)} 
x \left(\frac{1}{ \eu_0(\fu^-_{ 1^d})\eu_0( \fu_{1^d} \otimes \CC_h ) } \right) [xs_aS_{1^d},xS_{1^d}]
\end{align}
For $u \in S_d$, we therefore compute:
\begin{align}
  \label{eq:64}
[Z_{1^d,\f_a1^d}] \star [Z_{\f_a 1^d,1^d}] \star [u] =  \\  
u \left( \frac{\alpha_a+h}{-\alpha_a} \right) [u] + 
u \left( \frac{\alpha_a+h}{\alpha_a} \right) [us_a] = \\ 
-[u] +[us_a] + \frac{h}{u(\alpha_a)} \left( [us_a] - [u] \right) = \\
(s_a - 1 + h \del_a) [u]
\end{align}
where $\del_a$ is just notation for the usual Bernstein-Gelfand-Gelfand operator.

Specializing $h=0$, we have:
\begin{align}
  \label{eq:65}
[Z_{1^d,\f_a1^d}] \star [Z_{\f_a 1^d,1^d}] \star [u] \vert_{h=0} = (s_a - 1) [u]      
\end{align}
Therefore we have computed $E_a F_a \equiv 1 - s_a$, and we can calculate the Weyl group action to be:
\begin{align}
  \label{eq:0327}
  T_a & \equiv s_a
\end{align}
We know that this action on $H_\bullet(\cF_{1^d})$ restricts to the Springer action on $H_\top(\tcN^x_{1^d})$, therefore we have proved the following.
\begin{Theorem}
  \label{thm:ginzburg-zero-weight-equals-springer}
 Let $\lambda \vdash d$, $x \in \OO_\lambda \subset \cN_d$. Then the Weyl group action on $H_\top(\tcN^x_{1^d})$ arising from the Ginzburg construction agrees with the Weyl group action coming from Springer theory.
\end{Theorem}

\begin{Remark}
  Consider the usual Steinberg variety $\cZ = \tcN_{1^d} \times_{\cN_d}   \tcN_{1^d}$. The irreducible components of $\cZ$ are naturally indexed by the Weyl group $S_d$; for each $w \in S_d$, let $Z_w$ be the corresponding component. The above calculation shows that
  \begin{align}
    \label{eq:66}
[Z_{s_a}] = [Z_{1^d,\f_a1^d}] \star [Z_{\f_a 1^d,1^d}]    
  \end{align}
  as classes in $H_\bullet^{T \times \GG_m}(\cZ)$. Formula \eqref{eq:63} shows that the classes agree after tensoring with the fraction field of $H_\bullet^{T\times \GG_m} (\pt)$, but by \eqref{eq:64} we can conclude that they agree prior to tensoring because the action of $H_\bullet^{T \times \GG_m}(\cZ)$ on $H_\bullet^{T \times \GG_m}(\cF_{1^d})$ is known to be faithful.
\end{Remark}

\section{Further analysis of the Braverman-Gaitsgory action}

\subsection{Some preliminaries on Schur-Weyl duality}
\label{sec:schur-weyl-duality-the-action-of-chevalley-generators}

\subsubsection{Invariants vs. coinvariants}
\newcommand{\triv}{\mathrm{triv}}
\newcommand{\G}{\fG}

Let $\G$ be a group, and let $V$ be a $\G$-representation. Then we can form the $\G$-invariants $V^\G$, which comes via a canonical map $V^\G \hookrightarrow V$.
Similarly, we have the $\G$-coinvariants $V_\G$, which we can realize as  
the quotient of $V$ by the span of $\left\{gv - v \suchthat g \in \G \textand v \in V\right\}$.

We always have the canonical map $V^\G \rightarrow V_\G$ defined as the composition:
\begin{align}
  \label{eq:0302}
  V^\G \hookrightarrow V \twoheadrightarrow V_\G
\end{align}
For us $\G$ is a finite group, and we work over a field where $\# \G$ is invertible (namely $\CC$). In this case, the map $V^\G \rightarrow V_\G$ is an isomorphism, and we have explicit inverse $V_\G \rightarrow V^\G$ given by:
\begin{align}
  \label{eq:0303}
 [v] \mapsto  \frac{1}{\#\G}\sum_{g \in \G} gv
\end{align}
In the discussion below, the more natural object will sometimes be coinvariants, and we will use map \eqref{eq:0303} to canonically identification coinvariants with invariants,

\subsubsection{Schur-Weyl duality and taking invariants}

Let $V$ be an $S_d$-module. Consider the $S_d \times \GL(\E)$-module $V \otimes \E^{\otimes d}$.
We will explicitly describe a $T$-equivariant isomorphism:
\begin{align}
  \label{eq:0270}
  \bigoplus_{\dd \in \cP_{n,d}}  V^{S_\dd} \otimes \E^{\dd} = ( V \otimes \E^{\otimes d})^{S_d} 
\end{align}
To describe this map, we need to define for each $\dd \in \cP_{n,d}$ and injective map
\begin{align}
  \label{eq:61}
  V^{S_\dd} \otimes \E^{\dd} \hookrightarrow ( V \otimes \E^{\otimes d})^{S_d} 
\end{align}

Note that we have a natural injection $V^{S_\dd} \otimes \E^{\dd} \hookrightarrow V \otimes \E^{\dd} \hookrightarrow V \otimes \E^{\otimes d}$. But the image of this map does not lie in the $S_d$-invariants. Instead we will further map to the coinvariants $(V \otimes \E^{\otimes})_{S_d}$ and then map to the invariants $(V \otimes \E^{\otimes})^{S_d}$ via the canonical section \eqref{eq:0303}. Thus we have defined an explicit map realizing \eqref{eq:0270}.

\subsubsection{The action of Chevalley generators}
Observer that $\GL(\E)$ acts on the right hand side of \eqref{eq:0270}. Therefore, we can transport this action to the left hand side of \eqref{eq:0270}, and below we will write formulas for the action of Chevalley generators.

For each $i \in [n]$, choose a basis vector $\theta_i \in \E_i$. For each $\dd \in \cP_{n,d}$, let $\theta_\dd$ be the corresponding basis vector in $E^\dd$. We need to make this choice to define Chevalley generators. For each $a \in [n-1]$, we define operators $E_a$ and $F_a$ by 
\begin{align}
  \label{eq:283}
  E_a(\theta_{a+1}) = \theta_a \\
  E_a(\theta_{i}) = 0 \text{ otherwise } 
\end{align}
and: 
\begin{align}
  \label{eq:284}
  F_a(\theta_{a}) = \theta_{a+1} \\
  F_a(\theta_{i}) = 0 \text{ otherwise } 
\end{align}

Let $\dd \in \cP_{n,d}$, and consider the case when $\e_a(\dd) \neq \nabla$. The operator $E_a$ acts sending the $\dd$-weight space of $(V \otimes E^{\otimes d})^{S_d}$ to its $\e_a(\dd)$-weight space. Transporting structure via  \eqref{eq:0270}, we get an operator $E_a : V^{S_{\dd}} \otimes E^{\dd} \rightarrow V^{S_{\e^a\dd}} \otimes E^{\e^a\dd}$. Because we have trivialized the lines $E^{\dd}$ and $E^{\e^a \dd}$, we therefore get an operator $E_a : V^{S_{\dd}}  \rightarrow V^{S_{\e^a\dd}}$. Similarly, we can construct operators $F_a$. The following is an explicit description of these operators. 

\begin{Proposition}
  \label{prop:chevalley-gen-action-via-schur-weyl-duality}
Let $\dd = (d_1,\ldots,d_n) \in \cP_{n,d}$, and consider the case when $\e_a\dd \neq \nabla$.  The operator  $E_a : V^{S_{\dd}}  \rightarrow V^{S_{\e^a\dd}}$ is given by the symmetrization operator
  \begin{align}
    \label{eq:0287}
   d_{a+1} \cdot \Psi_{\e^a \dd, \dd} :  V^{S_{\dd}}  \rightarrow V^{S_{\e^a\dd}}
  \end{align}
  where $\Psi_{\e^a \dd, \dd}$ defined by:
  \begin{align}
    \label{eq:0288}
    \Psi_{\e^a \dd, \dd}( v) = \frac{1}{\# S_{\e_a \dd}} \sum_{w \in S_{\e_a \dd} } w(v)
  \end{align}
  In the case when $\f_a \dd \neq \nabla$, the operator  $F_a : V^{S_{\dd}}  \rightarrow V^{S_{\f^a\dd}}$ is given by $d_{a} \cdot \Psi_{\f^a \dd, \dd} $.
\end{Proposition}

\begin{proof}
  We do the case of $E_a$. The case of $F_a$ is essentially the same.
  Let $v \in V^{S_\dd}$. Under \eqref{eq:0270},  this maps to the $S_d$-invariant:
  \begin{align}
    \label{eq:0289}
    \frac{1}{\# S_d} \sum_{x \in S_d} x(v) \otimes x(\theta_{\dd})
  \end{align}
  Writing out $\dd = (d_1,\ldots,d_n)$, we calculate that
  \begin{align}
    \label{eq:0290}
    E_{a}(\theta_{\dd}) = \sum_{j=1}^{d_{a+1}} (d_1+\cdots+d_a +1, d_1+\cdots+d_a +j) \theta_{\e_a \dd}
  \end{align}
  where $(d_1+\cdots+d_a +1, d_1+\cdots+d_a +j)$ is the transposition. So $E_av$ corresponds to:
  \begin{align}
    \label{eq:0291}
    \frac{1}{\# S_d} \sum_{x \in S_d} \sum_{j=1}^{d_{a+1}} x(v) \otimes x( (d_1+\cdots+d_a +1, d_1+\cdots+d_a +j)\theta_{\e_a\dd})
  \end{align}
  Note that for each $j$, the transposition $(d_1+\cdots+d_a +1, d_1+\cdots+d_a +j) \in S_{\dd}$. So we can rewrite as:
  \begin{align}
    \label{eq:0294}
   & \frac{1}{\# S_{d}}  \sum_{x \in S_d} \sum_{j=1}^{d_{a+1}} x((d_1+\cdots+d_a +1, d_1+\cdots+d_a +j)v) \otimes x( (d_1+\cdots+d_a +1, d_1+\cdots+d_a +j)\theta_{\e_a\dd}) \\
    &= \frac{d_{a+1}}{\# S_{d}}  \sum_{x \in S_d}  x(v) \otimes x( \theta_{\e_a\dd})
  \end{align}
  We can write:
  \begin{align}
    \label{eq:0296}
    \frac{d_{a+1}}{\# S_{d}}  \sum_{x \in S_d}  x(v) \otimes x( \theta_{\e_a\dd}) = 
    \frac{d_{a+1}}{\# S_{\e_a\dd}\# S_{d}}  \sum_{x \in S_d} \sum_{w \in S_{\e_a\dd}}  x(w(v)) \otimes x(w( \theta_{\e_a\dd})) 
  \end{align}
  Because $\theta_{\e_a\dd}$ is $S_{\e_a\dd}$-invariant, this is equal to:
  \begin{align}
    \label{eq:0297}
    \frac{d_{a+1}  }{\# S_{d}}  \sum_{x \in S_d}    x(\Psi_{\e_a\dd,\dd}(v)) \otimes x(\theta_{\e_a\dd})
  \end{align}
  Under \eqref{eq:0270}, this corresponds to  
  \begin{align}
    \label{eq:0298}
    d_{a+1} \cdot \Psi_{\e_a\dd,\dd}(v)
  \end{align}
  as an element of $V^{S_{\e_a \dd}}$.  
\end{proof}

\subsubsection{Dual construction}

Above we started with an $S_d$-module $V$ and produced a $\GL(\E)$ module. We can also consider the dual $S_d$-module $V^*$ and run the above procedure to obtain an 
$\GL(\E)$-module. As $S_d$-modules are self-dual, we know that the resulting module is the same as an abstract $\GL(\E)$-module. We will describe explicitly the Chevalley generators.

For each $\dd \in \cP^{++}_{n,d}$, the $\dd$-weight space is the invariant space $(V^*)^{S_\dd}$, which is canonically isomorphic to the dual of the coinvariant space $(V_{S_\dd})^*$. Using \eqref{eq:0303}, we can identify this with the dual of the invariant space $(V^{S_\dd})^*$. Therefore, we obtain Chevalley generators for each $a \in [n-1]$
\begin{align}
  \label{eq:118}
  E_a : (V^{S_\dd})^* \rightarrow (V^{S_{\e_a\dd}})^*
\end{align}
and
\begin{align}
  \label{eq:119}
  F_a : (V^{S_\dd})^* \rightarrow (V^{S_{\f_a\dd}})^*
\end{align}
under the appropriate assumptions that $\e_a \dd \neq \nabla$ and $\f_a \dd \neq \nabla$.
Unwinding Proposition \ref{prop:chevalley-gen-action-via-schur-weyl-duality} and using \eqref{eq:0303} to identify invariants and coinvariants, we obtain the following proposition.
\begin{Proposition}
  \label{prop:dual-chevalley-generators}
  Let $a \in [n-1]$ and $\dd = (d_1,\ldots,d_n) \in \cP^{++}_{n,d}$. In the case when $\e_a\dd \neq \nabla$, the Chevalley generator $E_a$ \eqref{eq:118} is the adjoint of the map:
  \begin{align}
    \label{eq:120}
   \frac{d_{a+1}}{d_{a}+1}F_a :  V^{S_{\e_a \dd}} \rightarrow V^{S_{\dd}}
  \end{align}
  When $\f_a \dd \neq \nabla$, the Chevalley generator $F_a$ \eqref{eq:119} is the adjoint of the map:
  \begin{align}
    \label{eq:121}
   \frac{d_a}{d_{a+1}+1}E_a :  V^{S_{\f_a \dd}} \rightarrow V^{S_{\dd}}
  \end{align}
\end{Proposition}

\subsection{Realizing Chevalley generators in the Braverman-Gaitsgory construction}

Recall that a key step in the Braverman-Gaitsgory construction is Theorem \ref{thm:brav-gaits-t-equiv-isom}, which we recall is the isomorphism:
\begin{align}
  \label{eq:67}
   (\Gro \otimes \E^{\otimes d})^{S_d} \overset{\sim}{\rightarrow}  (p_{n,d})_*\cK^{n,d}_\E
\end{align}
Roughly speaking, \eqref{eq:67} is a sheaf-theoretic version of \eqref{eq:0270}. Below we will explicitly describe this isomorphism in detail.

\subsubsection{Explicit description of  \eqref{eq:67}}
Write
\begin{align}
  \label{eq:0254}
 \pi_{\dd,1^d} : \tfg_{1^d} \rightarrow \tfg_{\dd} 
\end{align}
for the natural map. Notice that $\pi_{\dd,1^d}$ is proper.
Recall that
\begin{align}
  \label{eq:0255}
\cK^{n,d}_\E \Big\vert_{\tfg_{\dd}} = \E^\dd[d^2]
\end{align}
by definition.
Because both $\tfg_{1^d}$ and  $\tfg_{\dd}$ are both smooth of dimension $d^2$, we have:
\begin{align}
  \label{eq:0257}
   \E^\dd[d^2] = \pi_{\dd,1^d}^! \E^\dd[d^2] 
\end{align}
By the $\left( (\pi_{\dd,1^d})_!,\pi_{\dd,1^d}^! \right)$-adjunction, we obtain a map
\begin{align}
  \label{eq:0258}
     (\pi_{\dd,1^d})_! \E^\dd[d^2] \rightarrow  \E^\dd[d^2] 
\end{align}
of sheaves on $\fg_\dd$.
Because $\pi_{\dd,1^d}$ is proper, we equivalently have a map:
\begin{align}
  \label{eq:0259}
     (\pi_{\dd,1^d})_* \E^\dd[d^2] \rightarrow  \E^\dd[d^2] 
\end{align}
Applying $(p_{\dd})_*$, we obtain a map
\begin{align}
  \label{eq:0260}
     \Gro \otimes \E^\dd \rightarrow  (p_{\dd})_* \E^\dd[d^2]
\end{align}
which on fibers exactly corresponds to the proper pushforward in Borel-Moore homology from the fibers of $p_{1^d}$ to the fibers of $p_{\dd}$.
Summing over $\dd \in \cP_{n,d}$, we obtain a map:
\begin{align}
  \label{eq:0261}
     \bigoplus_{\dd \in \cP_{n,d}} \Gro \otimes \E^\dd \rightarrow  (p_{n,d})_* \cK^{n,d}_\E
\end{align}

Notice that for each $\dd$, we have the map:
\begin{align}
  \label{eq:0262}
 \Gro^{S_\dd} \otimes \E^\dd  \rightarrow \Gro \otimes \E^\dd
\end{align}
which sums to a map:
\begin{align}
  \label{eq:0263}
\bigoplus_{\dd \in \cP_{n,d}} \Gro^{S_\dd} \otimes \E^\dd  \rightarrow \bigoplus_{\dd \in \cP_{n,d}} \Gro \otimes \E^\dd
\end{align}
We can consider the composed map:
\begin{align}
  \label{eq:0264}
\bigoplus_{\dd \in \cP_{n,d}} \Gro^{S_\dd} \otimes \E^\dd  \rightarrow \bigoplus_{\dd \in \cP_{n,d}} \Gro \otimes \E^\dd
\rightarrow  (p_{n,d})_* \cK^{n,d}_\E  
\end{align}
By the smallness of the maps $p_{n,d}$ and the $p_{1^d}$, all the sheaves in \eqref{eq:0264} are Goresky-MacPherson extensions of their restrictions to the open set $\fg^{rs}_d$  of regular semisimple elements. On the regular semi-simple locus, the composed map \eqref{eq:0264} is an isomorphism by the discussion in \cite[\S 2.6]{Braverman-Gaitsgory}. Therefore, by the Perverse Continuation Principle, the composed map \eqref{eq:0264} is an isomorphism on all of $\fg_d$.
Recall, that by \eqref{eq:61}, we have a canonical isomorphism:
\begin{align}
  \label{eq:0256}
\bigoplus_{\dd \in \cP_{n,d}} \Gro^{S_\dd} \otimes \E^\dd  \overset{\sim}{\rightarrow} (\Gro \otimes \E^{\otimes d})^{S_d}
\end{align}
Noting that $(\Gro \otimes \E^{\otimes d})^{S_d}$ has a $\GL(\E)$-action, we can transport structure to obtain a $\GL(\E)$-action on $\bigoplus_{\dd \in \cP_{n,d}} \Gro^{S_\dd} \otimes \E^\dd$

\subsubsection{The $\GL(\E)$-action on stalks}
Let $\lambda \vdash d$. Let $\OO_\lambda \subset \fg_{d}$ be the corresponding nilpotent orbit and let $x \in \OO_\lambda$.

The $!$-stalks of $\Gro$ and $(p_{\dd})_* \E^\dd[d^2]$ are equal to the Borel-Moore homologies $H_\bullet(\tfg^x_{1^d})$ and $H_\bullet(\tfg^x_{\dd})$ respectively.
The $S_d$-action on $\Gro$ gives rise to an $S_d$-action $H_\bullet(\tfg^x_{1^d})$. This is given by the usual Springer $S_d$-action on $H_\bullet(\tfg^x_{1^d})$ (see e.g. \cite[Ch. 13]{Jantzen}).

Recall that the map \eqref{eq:0261} corresponds to the proper pushforward from $H_\bullet(\tfg^x_{1^d})$ to $H_\bullet(\tfg^x_{\dd})$. Therefore, two maps in \eqref{eq:0264}, induce the maps:
\begin{align}
  \label{eq:0266}
 H_\bullet(\tfg^x_{1^d})^{S_{\dd}} \hookrightarrow H_\bullet(\tfg^x_{1^d}) \rightarrow H_\bullet(\tfg^x_{\dd})
\end{align}
As the composed map in \eqref{eq:0264} is an isomorphism, the composed map
\begin{align}
  \label{eq:0267}
 H_\bullet(\tfg^x_{1^d})^{S_{\dd}} \overset{\sim}{\rightarrow} H_\bullet(\tfg^x_{\dd})
\end{align}
is an isomorphism.
Summing over all $\dd$ and tensoring with $\E^{\dd}$ we have an isomorphism:
\begin{align}
  \label{eq:0268}
 \bigoplus_{\dd \in \cP_{n,d}} H_\bullet(\tfg^x_{1^d})^{S_{\dd}} \otimes \E^\dd \overset{\sim}{\rightarrow} \bigoplus_{\dd \in \cP_{n,d}} H_\bullet(\tfg^x_{\dd}) \otimes \E^\dd
\end{align}
The subrepresentation
\begin{align}
  \label{eq:0272}
 \bigoplus_{\dd \in \cP_{n,d}} H_{2d_\lambda} (\tfg^x_{1^d})^{S_{\dd}} \otimes \E^\dd \overset{\sim}{\rightarrow} \bigoplus_{\dd \in \cP_{n,d}} H_{2d_\lambda}(\tfg^x_{\dd}) \otimes \E^\dd
\end{align}
is isomorphic to the irreducible $\GL(\E)$-module $V(\lambda)$.  The $d_\lambda$-dimensional irreducible components of $\tfg^x_{\dd}$ form a basis of $H_{2d_\lambda}(\tfg^x_{\dd})$. 

Note that $\tfg^0_{1^d} = \cF_{1^d}$, and we have an inclusion $\tfg^x_{1^d} \hookrightarrow \cF_{1^d}$. 
By the argument in the proof of Theorem 6.5.2(b) in Chriss-Ginzburg, the map $H_\bullet(\tfg^x_{1^d}) \rightarrow H_\bullet(\cF_{1^d})$ is $S_d$-equivariant.
Therefore, the map
\begin{align}
  \label{eq:0269}
  \bigoplus_{\dd \in \cP_{n,d}} H_\bullet(\tfg^x_{1^d})^{S_{\dd}} \otimes E^\dd \rightarrow \bigoplus_{\dd \in \cP_{n,d}} H_\bullet(\cF_{1^d})^{S_{\dd}} \otimes E^\dd
\end{align}
is $\GL(E)$-equivariant. Similarly, the map
\begin{align}
  \label{eq:0271}
\bigoplus_{\dd \in \cP_{n,d}} H_\bullet(\tfg^x_{\dd}) \otimes E^\dd  \rightarrow \bigoplus_{\dd \in \cP_{n,d}} H_\bullet(\cF_{\dd}) \otimes E^\dd
\end{align}
is $\GL(E)$-equivariant.
By the irreducibility of $V(\lambda)$, the map
\begin{align}
  \label{eq:0273}
  \bigoplus_{\dd \in \cP_{n,d}} H_{2d_\lambda}(\tfg^x_{\dd}) \otimes E^\dd \hookrightarrow \bigoplus_{\dd \in \cP_{n,d}} H_\bullet(\cF_{\dd}) \otimes E^\dd
\end{align}
is injective (see the proof of Theorem 6.5.2(a) in Chriss-Ginzburg). We summarize this discussion as the following proposition.

\begin{Proposition}
  Let $\lambda \vdash d$, and let $x \in \OO_\lambda$. Let $d_\lambda = \dim \cF_{1^d} - \frac{1}{2}\dim \OO_{\lambda} = \sum \lambda_i (i-1)$. Then all the maps in the following commutative square are $\GL(E)$-equivariant.
  \begin{equation}
    \label{eq:0274}
    \begin{tikzcd}
  \bigoplus_{\dd \in \cP_{n,d}} H_{2d_\lambda}(\tfg^x_{1^d})^{S_{\dd}} \otimes \E^\dd  \arrow[d,"\sim"] \arrow[r,hookrightarrow] & \bigoplus_{\dd \in \cP_{n,d}} H_\bullet(\cF_{1^d})^{S_{\dd}} \otimes \E^\dd \arrow[d,"\sim"] \\
  \bigoplus_{\dd \in \cP_{n,d}} H_{2d_\lambda}(\tfg^x_{\dd}) \otimes \E^\dd \arrow[r,hookrightarrow] & \bigoplus_{\dd \in \cP_{n,d}} H_\bullet(\cF_{\dd}) \otimes \E^\dd 
    \end{tikzcd}
  \end{equation}
The horizontal maps are inclusions, and the vertical maps are isomorphisms.  
\end{Proposition}

\subsubsection{Chevalley generators}

We follow the discussion and notation in section  \ref{sec:schur-weyl-duality-the-action-of-chevalley-generators}. We have our basis element $\theta_\dd \in E^\dd$ for each $\dd \in \cP_{n,d}$. Let $\dd \in \cP_{n,d}$. Consider $a \in [n-1]$ so that $\e_a \dd \neq \nabla$. Then we have the operator $E_a = d_{a+1} \Psi_{\e_a\dd,\dd} : H_\bullet(\cF_{1^d})^{S_\dd} \rightarrow H_\bullet(\cF_{1^d})^{S_{\e_a\dd}}$. 
We have a commutative diagram:
\begin{equation}
  \label{eq:0304}
    \begin{tikzcd}
    H_\bullet(\cF_{1^d})^{S_{\dd}}  \arrow[d,"\sim"] \arrow[r,"E_a"] &  H_\bullet(\cF_{1^d})^{S_{\e_a\dd}}  \arrow[d,"\sim"] \\
    H_\bullet(\cF_{\dd})  \arrow[r,"E_a"] &  H_\bullet(\cF_{\e_a\dd}) 
    \end{tikzcd}
\end{equation}

The $S_d$-action on $H_\bullet(\cF_{1^d})$ is given by restricting the $S_d$-action on $H^T_\bullet(\cF_{1^d})$, so 
formula \eqref{eq:0288} defining $\Psi_{\e_a\dd,\dd}$ has a natural equivariant lift to an operator:
\begin{align}
  \label{eq:0316}
  \Psi_{\e_a\dd,\dd} : H_\bullet(\cF_{1^d})^{S_{\dd}}  \rightarrow  H_\bullet(\cF_{1^d})^{S_{\e_a\dd}} 
\end{align}
Therefore the commutative square \eqref{eq:0304} has a natural $T$-equivariant lift:
\begin{equation}
  \label{eq:0317}
    \begin{tikzcd}
    H_\bullet^T(\cF_{1^d})^{S_{\dd}}  \arrow[d,"\sim"] \arrow[r,"E_a"] &  H_\bullet^T(\cF_{1^d})^{S_{\e_a\dd}}  \arrow[d,"\sim"] \\
    H_\bullet^T(\cF_{\dd})  \arrow[r,"E_a"] &  H_\bullet^T(\cF_{\e_a\dd}) 
    \end{tikzcd}
\end{equation}
The $S_d$-action on $H^T_\bullet(\cF_{1^d})$ commutes with the action of $H^T_\bullet(\pt)$, and we therefore conclude that all the maps in \eqref{eq:0317} are equivariant for the action of $H^T_\bullet(\pt)$.

We will compute the how the operators $E_a$ act on torus fixed points, which by the localization formula, will determine the operators. Let $w \in S_d$, and let $\frac{1}{\#S_{\dd}}\sum_{x \in S_{\dd}}  [wx] \in H_\bullet(\cF_{1^d})^{S_{\dd}}$. Under proper pushforward to $H_\bullet(\cF_{\dd})$:
\begin{align}
  \label{eq:0305}
\frac{1}{\#S_{\dd}} \sum_{x \in S_{\dd}}  [wx]  \mapsto [w S_{\dd}]
\end{align}
Under $E_a : H_\bullet(\cF_{1^d})^{S_{\dd}}  \rightarrow H_\bullet(\cF_{1^d})^{S_{\e_a\dd}}$, by Proposition  \ref{prop:chevalley-gen-action-via-schur-weyl-duality}:
\begin{align}
  \label{eq:0306}
\frac{1}{\#S_{\dd}} \sum_{x \in S_{\dd}}  [wx]  \mapsto \frac{d_{a+1}}{\# S_{\dd} \# S_{\e_a \dd}} \sum_{x \in S_{\dd}}\sum_{y \in S_{\e_a\dd}} [wxy]
\end{align}
Under proper pushforward to $H_{\bullet}(\cF_{\e_a\dd})$:
\begin{align}
  \label{eq:0299}
\frac{d_{a+1}}{\# S_{\dd} \# S_{\e_a \dd}} \sum_{x \in S_{\dd}}\sum_{y \in S_{\e_a\dd}} [wxy] \mapsto
\frac{d_{a+1}}{\# S_{\dd} } \sum_{x \in S_{\dd}} [wxS_{\e_a\dd}] 
\end{align}
Therefore, we have the following.

\begin{Proposition}
  Let $\dd \in \cP_{n,d}$, and let $wS_{\dd} \in S_{d}/S_{\dd}$. Let $a \in [n-1]$. When $\e_a \dd \neq \nabla$,
  the operator 
  \begin{align}
    E_a : H_\bullet(\cF_{\dd})  \rightarrow  H_\bullet(\cF_{\e_a\dd}) 
  \end{align}
  satisfies:  
  \begin{align}
    E_a([wS_\dd]) = \frac{d_{a+1}}{\# S_{\dd} } \sum_{x \in S_{\dd}} [wxS_{\e_a\dd}] 
  \end{align}
  Similarly, when $\f_a \dd \neq \nabla$, 
  the operator 
  \begin{align}
    F_a : H_\bullet(\cF_{\dd})  \rightarrow   H_\bullet(\cF_{\f_a\dd}) 
  \end{align}
  satisfies:  
  \begin{align}
    F_a([wS_\dd]) = \frac{d_{a}}{\# S_{\dd} } \sum_{x \in S_{\dd}} [wxS_{\f_a\dd}] 
  \end{align}
\end{Proposition}

\begin{Remark}
  Analagous formulas for the action of Chevalley generators on the cohomology of little Spaltenstein varieties have been given by Brundan, Ostrik, and Vasserot \cite{Brundan,Brundan-Ostrik,Vasserot-1993}. 
\end{Remark}

\subsubsection{Chevalley generators as correspondences}

Let $\dd, \dd' \in \cP_{n,d}$. Recall that we can form the ``Steinberg'' variety $\tcN_{\dd'} \times_{\cN_{n,d}} \tcN_{\dd}$. The component $Z_{\dd',\dd}$, which we initially defined as the conormal bundle of the subvariety $Y_{\dd',\dd} \subseteq \cF_{\dd'} \times \cF_{\dd}$, can also be defined as the following fiber product:
\begin{equation}
  \label{eq:0310}
\begin{tikzcd}
   Z_{\dd',\dd} \arrow[d] \arrow[r] \arrow[dr, phantom, "\square"] & \tcN_{\dd'} \times_{\cN_{n,d}} \tcN_{\dd} \arrow[d] \\
   Y_{\dd',\dd} \arrow[r,hookrightarrow] &   \cF_{\dd'} \times \cF_{\dd} 
\end{tikzcd}
\end{equation}

Similarly, on the ``$\tfg$''-side, we can form the ``Steinberg'' variety $\tfg_{\dd'} \times_{\fg_d} \tfg_{\dd}$. Analagous to $Z_{\e_a \dd,\dd}$, we can form the subvariety $X_{\dd',\dd}$ that is defined as the following fiber product:
\begin{equation}
  \label{eq:0309}
  \begin{tikzcd}
   X_{\dd',\dd} \arrow[d] \arrow[r] \arrow[dr, phantom, "\square"] & \tfg_{\dd'} \times_{\fg_d} \tfg_{\dd} \arrow[d] \\
   Y_{\dd',\dd} \arrow[r,hookrightarrow] &   \cF_{\dd'} \times \cF_{\dd} 
 \end{tikzcd}
\end{equation}

We have
\begin{align}
  \label{eq:0313}
  Y_{\dd',\dd} = G/{P_{\dd'} \cap P_{\dd}} 
\end{align}
and:
\begin{align}
  \label{eq:0311}
  Z_{\dd',\dd} = G \times^{P_{\dd'} \cap P_{\dd}} (\fu_{\dd} \cap \fu_{\dd'})
\end{align}
Similarly, we have:
\begin{align}
  \label{eq:0312}
  X_{\dd',\dd} = G \times^{P_{\dd'} \cap P_{\dd}} (\fp_{\dd} \cap \fp_{\dd'})
\end{align}
From this, we compute
\begin{align}
  \label{eq:0314}
  [X_{\dd',\dd}] = \sum_{x \in S_{d} / (S_{\dd'} \cap S_{\dd})} x \left(\frac{1}{ \eu_0(\fu^-_{\dd'}+ \fu^-_\dd)\eu_0( (\fp_{\dd'} \cap \fp_\dd) \otimes \CC_h )} \right) [xS_{\dd'},xS_{\dd}]
\end{align}
as classes in localized equivariant homology.

\begin{Theorem}
  \label{thm:bg-chevalley-by-explicit-correspondences}
  Let $\dd \in \cP_{n,d}$, and let $c \in H_\bullet(\cF_\dd)$.  Let $a \in [n-1]$. When $\e_a \dd \neq \nabla$, we have:
  \begin{align}
    \label{eq:0315}
   [X_{\e_a\dd,\dd}] \star  c = E_a(c)
  \end{align}
When $\f_a \dd \neq \nabla$, we have:
  \begin{align}
    \label{eq:00315}
   [X_{\f_a\dd,\dd}] \star  c = F_a(c)
  \end{align}
\end{Theorem}

\begin{proof}
  We will prove \eqref{eq:0315}. Equation \eqref{eq:00315} is similar. We will show that $~\eqref{eq:315}$ holds $T$-equivariantly. Initially, we will work $T \times \GG_m$-equivariantly. By the localization, it suffices to consider the case when $c = [wS_{\dd}]$ is a torus-fixed point.
  We compute:
  \begin{align}
    \label{eq:0318}
    [X_{\e_a\dd,\dd}] \star  [wS_{\dd}] & = \sum_{y \in S_{\dd}/S_{\e_a\dd \cap S_{\dd}}} wy \left( \frac{\eu_0(\fu^-_\dd) \eu_0(\fp_\dd \otimes \CC_h) }{\eu_0(\fu_{\e_a\dd}^- + \fu_{\dd}^-) \eu_0( (\fp_{\e_a\dd} \cap \fp_{\dd}) \otimes \CC_h) }\right)[wyS_{\e_a\dd}] \\
  \end{align}
  Notice that both numerator and denominator have a factor of $h^n$. Therefore, we can specialize $h=0$ and compute:
  \begin{align}
    \label{eq:0319}
    [X_{\e_a\dd,\dd}] \star  [wS_{\dd}]\big\vert_{h=0} & = \sum_{y \in S_{\dd}/(S_{\e_a\dd} \cap S_{\dd})} [wyS_{\e_a\dd}] \\
                                                       & =  \frac{1}{\#(S_{\e_a\dd} \cap S_{\dd})} \sum_{y \in S_{\dd}} [wyS_{\e_a\dd}] \\
                                                         & = \frac{d_{a+1}}{\# S_{\dd}}\sum_{y \in S_{\dd}} [wyS_{\e_a\dd}] \\
                                                         & = E_a([wS_{\dd}])
  \end{align}
  
\end{proof}

\subsection{Action on weight zero spaces for small representations}

We will consider the analogue of the situation in \S~\ref{sec:ginz-weight-zero-small} for the Braverman-Gaitsgory action. Namely we will consider the case of $n=d$ for the remainder of this section. 
Fix $\lambda \vdash d$. Let $x \in \OO_\lambda \subset \cN_{d}$. We have an $\sl_d$-action on
\begin{align}
  \label{eq:68}
  \bigoplus_{\dd \in \cP_{d,d}} H_{2 d_\lambda}(\tfg^x_\dd)
\end{align}
that realizes the irreducible module $V(\lambda)$. The space
$H_{2 d_\lambda}(\tfg^x_{1^d})$ is precisely the zero weight space for the $\sl_d$-module. Recall that the operator $T_a$ acts as $1 - E_a F_a $. 

The following computation is a straightforward variation of the computation in \S \ref{subsubsect-computing-T-a}, which we will omit to save space.
\begin{Proposition}
  Let $a \in [d-1]$, as operators on $H^{T\times \GG_m}_{\bullet}(\tfg^0_{1^d})$ convolution by $[X_{1^d,\e_a 1^d}] \star [X_{\e_a 1^d,1^d}]$ is equal to:
  \begin{align}
    \label{eq:70}
    s_a + 1 + h \del_a
  \end{align}
\end{Proposition}
\noindent Further specializing $h=0$ and restricting attention to the top homology of $\tfg^x_{1^d}$ we have
\begin{align}
  \label{eq:71}
  E_aF_a =  s_a + 1
\end{align}
as operators on $H_{2 d_\lambda}(\tfg^x_{1^d})$
We therefore calculate:
\begin{align}
  \label{eq:72}
  T_a &\equiv - s_a
\end{align}
\begin{Theorem}
  \label{thm:bgv-weight-zero-springer-action}
  Fix $\lambda \vdash d$. Let $x \in \OO_\lambda \subset \cN_{d}$. The Braverman-Gaitsgory Weyl group actions on the zero weight spaces $H_{2 d_\lambda}(\tfg^x_{1^d})$ and $H^{2 d_\lambda}(\tfg^x_{1^d})$ are equal to the Springer actions tensored with the sign representation.
\end{Theorem}
\begin{proof}
 The statement in Borel-Moore homology is exactly what we have shown in the discussion above. To obtain the statement in cohomology, we use Proposition \ref{prop:dual-chevalley-generators} to compute that the Weyl group action on $H^{2 d_\lambda}(\tfg^x_{1^d})$ is exactly the dual action.
\end{proof}

\section{Mirkovi\'c-Vilonen cycles and orbital varieties}
\label{sec:MV-cycles-and-orbital-varieties}

\subsection{Lusztig's map and the first bijection between Mirkovi\'c-Vilonen cycles and orbital varieties}

In this section we will use the result of Braverman, Gaitsgory, and Vybornov (Theorem \ref{thm:bgv-comparison-with-mv-basis}) and our Theorem \ref{thm:bgv-weight-zero-springer-action} to compare Weyl group actions on MV cycles and orbital varieties.

Let us write $\Gr$ for the affine Grassmannian of $\SL_d$. Let $\Gr_0$ be the big cell of $\SL_d$, that is, the $\SL_d[t^{-1}]$-orbit of the unit point. Let $\lambda \leq d \omega_1$ be a dominant weight. Recall that we can consider $\lambda$ as a partition of $d$. Then we have the corresponding $SL_d(\cO)$-orbit closure $\overline{\Gr^\lambda}$ and the $U(\cK)$-orbit $S_0$. Following Mirkovi\'c and Vilonen, we can identify 
the compactly supported cohomology $H_c^\top( \overline{\Gr^\lambda} \cap S_0)$ with identified with the $0$-weight space of $V(\lambda)$.

Recall also that there is $\SL_d$-equivariant embedding 
\begin{align}
  \label{eq:73}
j:\cN_d \hookrightarrow \Gr
\end{align}
originally due to Lusztig \cite{Lusztig-1981} and defined as follows. Given $x \in \cN_d$, we can form the element $1 - t^{-1} x \in SL_d[t^{-1}]$, which we map to $\Gr$ by acting on the unit point.
The image of \eqref{eq:73} is precisely $ \overline{\Gr^{n\omega_1}} \cap \Gr_0$.
Under \eqref{eq:73}, the image of $\OO_\lambda$ is $ \Gr^{\lambda} \cap \Gr_0$, and the image of $\fu \subset \cN_d$ is $S_0$.

Therefore, $j$ induces an isomorphism:
\begin{align}
  \label{eq:69}
  j_\lambda : \overline{\OO_\lambda} \cap \fu \overset{\sim}{\rightarrow} \overline{\Gr^\lambda} \cap S_0
\end{align}
Therefore \emph{a fortiori} we have a bijection
\begin{align}
  \label{eq:107}
 j_\lambda :  \Irr(\overline{\OO_\lambda} \cap \fu) \overset{\sim}{\rightarrow} \Irr(\overline{\Gr^\lambda} \cap S_0)
\end{align}
The irreducible components of $\overline{\OO_\lambda} \cap \fu$ are usually called \terminology{orbital varieties}.

\begin{Remark}
Usually, one defines orbital varieties as closures in $\fu$ of irreducible components of $\OO_\lambda \cap \fu$. However, using the map $j_\lambda$, we can identify these with irreducible components of $\Gr^\lambda \cap S_0$. The work of Mirkovi\'c and Vilonen shows that the irreducible components of $\Gr^\lambda \cap S_0$ coincide with the irreducible components of $\overline{\Gr^\lambda} \cap S_0$. Applying $j_\lambda^{-1}$, we see that orbital varieties are precisely the irreducible components of $\overline{\OO_\lambda} \cap \fu$. We do not know if this holds more generally in other types (where the comparison map $j_\lambda$ to the affine Grassmannian does not exist).
\end{Remark}

\subsubsection{Identifying orbital varieties with Springer components}
\label{subsubsec:orb-vars-and-springer-components}
\newcommand{\Conv}{\mathrm{Conv}}

Let us briefly recall how one identifies orbital varieties with Springer components. Recall that we have the Springer resolution $\mu_{1^d} : \tcN_{1^d} \rightarrow \cN_d$. Let $x \in \OO_\lambda$. Because the centralizer of $x$ in $G = \GL_d$ is connected, we have a natural bijection between $\Irr(\tcN_{1^d})$ and $\Irr( (\mu_{1^d})^{-1}(\OO_\lambda))$. One can also realize
\begin{align}
  \label{eq:81}
   (\mu_{1^d})^{-1}(\OO_\lambda) = G \times^B (\OO_\lambda \cap \fu)
\end{align}
Again we have a natural bijection between  $\Irr(G \times^B (\OO_\lambda \cap \fu))$ and $ \Irr(\OO_\lambda \cap \fu)$. Finally, we have a bijection $ \Irr(\OO_\lambda \cap \fu)$ between $\Irr(
\overline{\OO_\lambda} \cap \fu)$. To summarize, we have constructed a bijection
\begin{align}
  \label{eq:108}
  s_\lambda : \Irr(\tcN_{1^d}) \overset{\sim}{\rightarrow} \Irr(\OO_\lambda \cap \fu)
\end{align}

\subsection{The Braverman-Gaitsgory-Vybornov construction and the second identification of MV cycles and orbital varieties}

In \cite{Braverman-Gaitsgory-Vybornov}, Braverman, Gaitsgory, and Vybornov (to be referred to as ``the authors'' for the remainder of this section for brevity) construct a bijection between MV cycles and big Spaltenstein components. In the special case of $\lambda \in \cP^{++}_{d,d}$, the big Spaltenstein components are usual Springer components. Combining this with the bijection $s_\lambda$ between Springer components and orbital varieties, we obtain another bijection:
\begin{align}
  \label{eq:109}
  \beta_\lambda :  \Irr(\overline{\OO_\lambda} \cap \fu) \overset{\sim}{\rightarrow} \Irr(\overline{\Gr^\lambda} \cap S_0)
\end{align}
We will show that this bijection agrees with the bijection $j_\lambda$ constructed above.

To aid the reader, we will following the notation in  \cite[\S\S 1,2]{Braverman-Gaitsgory-Vybornov} closely. We will refer the reader to their paper for background on the lattice model of $\GL_n$ affine Grassmannians. The authors define $E$ and $V$ are two distinct copies of $\CC^n$ that we canonically identify. The two vector spaces play different roles hence the reason for distinguishing them. Unfortunately, the vector space called $V$ in \cite{Braverman-Gaitsgory-Vybornov} corresponds to the vector space $E$ in \cite{Braverman-Gaitsgory}. This corresponds to the vector space we have been calling $\E$ elsewhere in the paper.

We will be interested in the case of $d=n$ and $\mu = 1^d = (1,\ldots,1)$ according to their notation.

\subsubsection{$\Gr_E$}

The $\Gr_E$ is the affine Grassmannian for $\GL(E)$, and $\Gr_E^{d,-}$ consists of lattices $\cM$ containing the standard lattice $\cM_0$ with $\dim \cM/\cM_0 = d$. We have $ \Gr_E^{-d \omega_1}$ where $-d \omega_1 = (-d,0,\ldots,0)$. Then we have an open embedding
\begin{align}
  \label{eq:87}
  j_E : \cN_d \hookrightarrow \Gr_E^{d,-}
\end{align}
given by sending a nilpotent matrix $x \in \cN_d$ to the lattice:
\begin{align}
  \label{eq:88}
  (1-x t^{-1})^{-1} t^{-1} \cM_0
\end{align}
Recall that because $x$ is a nilpotent $d \times d$ matrix, we have:
\begin{align}
  \label{eq:89}
  (1-x t^{-1})^{-1} = 1 + t^{-1}x + \cdots + t^{-(n-1)}x^{n-1}
\end{align}

There is also the space $\Conv^{1^d,-}(\Gr_E)$ which consists of sequences of lattices $(\cM_1,\ldots,\cM_d)$ with $\cM_{i-1} \subset \cM_{i}$ and $\dim \cM_{i}/\cM_{i-1} = 1$. We have an open embedding (after taking reduced scheme structures)
\begin{align}
  \label{eq:90}
 \tilde{j}_E :\tcN_{1_d} \hookrightarrow \Conv^{1^d,-}(\Gr_E)
\end{align}
sending a pair $(x,0 = F_0 \subset F_1 \subset \cdots \subset F_d = E )$
of nilpotent operator $x$ and an invariant flag $F_0 \subset F_1 \subset \cdots \subset F_d = E$ to the sequence of lattices $(\cM_1,\ldots,\cM_d)$ defined by:
\begin{align}
  \label{eq:91}
  \cM_i =  \cM_0 \oplus (1 - t^{-1} x)^{-1} t^{-1} F_i
\end{align}

The authors construct a Cartesian square of stacks
\begin{equation}
  \label{eq:94}
  \begin{tikzcd}
\Conv^{1^d,-}(\Gr_E) \arrow[r] \arrow[d] \arrow[dr, phantom, "\square"]& \tcN_{1^d} / \GL_d \arrow[d] \\    
\Gr^{d,-}_E \arrow[r] & \cN_{d} / \GL_d \\    
  \end{tikzcd}
\end{equation}
that they use to compare the Geometric Satake action to the Braverman-Gaitsgory action we have considered above.
This can be extended to the following two Cartesian squares
\begin{equation}
  \label{eq:95}
\begin{tikzcd}
\tcN_{1^d} \arrow[d] \arrow[r,"\tilde{j_E}"] \arrow[dr, phantom, "\square"] & \Conv^{1^d,-}(\Gr_E) \arrow[r] \arrow[d] \arrow[dr, phantom, "\square"]& \tcN_{1^d} / \GL_d \arrow[d] \\    
\cN_{d} \arrow[r,"j_E"] & \Gr^{d,-}_E \arrow[r] & \cN_{d} / \GL_d \\    
\end{tikzcd}
\end{equation}
where the composed horizontal maps are the tautological quotient maps.

\subsubsection{$\Gr_V$}

The $\Gr_V$ is the affine Grassmannian for $\GL(V)$, and $\Gr_V^{d,+}$ consists of lattices $\cM'$ contained in the standard lattice $\cM'_0$ with $\dim \cM'_0/\cM' = d$. We have $ \Gr_V^{d,+} = \Gr_V^{d \omega_1}$ where $d \omega_1 = (d,0,\ldots,0)$. Then we have an open embedding
\begin{align}
  \label{eq:087}
  j_V : \cN_d \hookrightarrow \Gr_V^{d}
\end{align}
given by sending a nilpotent matrix $x \in \cN_d$ to the lattice:
 \begin{align}
   \label{eq:088}
   (1-x t^{-1}) t \cM'_0
 \end{align}

\subsubsection{$\cP_{loc}$}

The authors define $\cP_{loc}$ to be the space of triples $(\cM,\cM',\alpha)$ where  $\cM \in \Gr_E^{d,-}$, $\cM' \in \Gr_V^{d,+}$, and $\alpha$ is a $\CC[[t]]$-equivariant isomorphism:
\begin{align}
  \label{eq:96}
  \alpha : \cM/\cM_0 \overset{\sim}{\rightarrow} \cM'_0/\cM'
\end{align}
We have projections $\pi_E : \cP_{loc} \rightarrow \Gr_E^{d,-}$ and 
$\pi_V : \cP_{loc} \rightarrow \Gr_V^{d,+}$.

We define the space $\cN_{loc}$ to be the set of triples $(x,y,g)$ where $x,y \in \cN_d$, $g \in \GL_d$, and $y = g x g^{-1}$. Then we have the  projections $\pi_1, \pi_2 : \cN_{loc} \rightarrow \cN_d$. Notice also that we have $\cN_{loc} \cong \GL_d \times \fu$ under the map $(x,y,g) \mapsto (y,g)$.

We define a map
\begin{align}
  \label{eq:92}
  b : \cN_{loc} \rightarrow \cP_{loc}
\end{align}
by sending $(x,y,g)$ to $(j_E(x),j_V(y),\alpha)$ where $\alpha$ is the isomorphism
\begin{align}
  \label{eq:93}
 \alpha : (1- t^{-1} x)^{-1} t^{-1}\cM_0 / \cM_0 \overset{\sim}{\rightarrow} \cM'_0/ (1- t^{-1} y) t\cM'_0
\end{align}
defined by 
\begin{align}
  \label{eq:97}
 \alpha( (1- t^{-1} x)^{-1} t^{-1} v) = g v  
\end{align}
for all $v \in \cM_0$. Here we identify $\cM_0$ and $\cM'_0$. Then it is easy to see that $b$ is an open embedding and that the following diagram is Cartesian:
\begin{equation}
  \label{eq:98}
  \begin{tikzcd}
   \cN_{loc} \arrow[d,"\pi_1 \times \pi_2"] \arrow[r,"b"] \arrow[dr, phantom, "\square"]&  \cP_{loc} \arrow[d,"\pi_E \times \pi_V"] \\
  \cN_d \times \cN_d \arrow[r,"j_E \times j_V"] &  \Gr^{d,-}_E \times \Gr^{d,+}_V
  \end{tikzcd}
\end{equation}

\subsection{Identifying MV cycles with Springer components}

Let $\cP_{loc}(1^d) = \pi_v^{-1}( \cap S_{1^d}) $. Note that $j_V^{-1}(S_{1^d}) = \fu$. Furthermore, the authors construct a factorization of the map $\pi_E :\cP_{loc}(1^d)  \rightarrow \Gr^{d,-}_E$ by a map:
\begin{align}
  \label{eq:99}
 \pi^{\Conv}_E : \cP_{loc}(1^d) \rightarrow \Conv^{1^d,-}(\Gr_E)
\end{align}
One can check that the image of $\pi^{\Conv}_E$ is contained in the image of $\tilde{j}_E$. Let us write $\Conv^{1^d,-}(\Gr_E)^\circ$ for the image.

We have a diagram: 
\begin{equation}
  \label{eq:101}
\begin{tikzcd}
 \cP_{loc}(1^d)  \arrow[r] \arrow[d] & \Gr^{d,+}_V \cap S_{1^d} \\
 \Conv^{1^d,-}(\Gr_E)^\circ
\end{tikzcd}
\end{equation}
Let $\cN_{loc}(1^d)$. Under the isomorphism $\cN_{loc} \cong \GL_d \times \fu$, we have $\cN_{loc}(1^d) \cong \GL_d \times \fu$. One can check that diagram \eqref{eq:101} is isomorphic to the diagram
\begin{equation}
  \label{eq:100}
\begin{tikzcd}
 G \times \fu  \arrow[r] \arrow[d] & \fu \\
 G \times^B \fu 
 \end{tikzcd}
\end{equation}
via the obvious isomorphisms.

Let $\lambda \in \cP^{++}_{d,d}$. Then define $\cP_{loc}^\lambda(1^d) = \pi_E^{-1}( \Gr^\lambda_V \cap S_{1^d}$. Let $\Conv^{1^d,-}(\Gr_E)^{\lambda,\circ}$ be the image of $\cP_{loc}^\lambda(1^d)$ under the map $\pi^{\Conv}_E$.
Then diagram \eqref{eq:101} restricts to a diagram:
\begin{equation}
  \label{eq:102}
\begin{tikzcd}
 \cP^\lambda_{loc}(1^d)  \arrow[r] \arrow[d] & \Gr^\lambda_V \cap S_{1^d} \\
\Conv^{1^d,-}(\Gr_E)^{\lambda,\circ} 
\end{tikzcd}
\end{equation}
The authors show that the arrows in \eqref{eq:102} are smooth with all fibers non-empty. Therefore, the set of irreducible components for each of the three spaces in \eqref{eq:102} are identified. The irreducible components of $\Gr^\lambda_V \cap S_{1^d}$ are exactly MV cycles of weight $(\lambda,1^d)$, and the irreducible components of $\Conv^{1^d,-}(\Gr_E)^{\lambda,\circ}$ are identified with components of a Springer fiber of type $\lambda$. Combining this with the map $\alpha_\lambda$ identifying Springer components with orbital varieties and the map identifying $\Gr^d_V$ with $\Gr_{\SL_d}$, we obtain a bijection:
\begin{align}
  \label{eq:110}
  \beta_\lambda :  \Irr(\overline{\OO_\lambda} \cap \fu) \overset{\sim}{\rightarrow} \Irr(\overline{\Gr^\lambda} \cap S_0)
\end{align}

Corresponding to the isomorphism between \eqref{eq:101} and \eqref{eq:100}, the restriction of \eqref{eq:102} is
\begin{equation}
  \label{eq:103}
\begin{tikzcd}
 G \times ( \OO_\lambda \cap \fu)   \arrow[r] \arrow[d] & \OO_\lambda \cap \fu \\
 G \times^B (\OO_\lambda \cap\fu )
 \end{tikzcd}
\end{equation}
Notice that we therefore obtain a bijection between $
\Irr(G \times^B (\OO_\lambda \cap\fu ))$ and $\Irr(\OO_\lambda \cap \fu)$, which is exactly the same as the bijection $s_\lambda$ considered above.  Again using the comparison between $\Gr_V^{d}$ and $\Gr_{\SL_d}$, we have proved the following.

\begin{Theorem}
  \label{thm:the-two-bijections-of-irreducible-components-agree}
The two bijections  $j_\lambda$ and $\beta_\lambda$  comparing $\Irr(\overline{\Gr^\lambda} \cap S_0)$ and $\Irr(\overline{\OO_\lambda} \cap \fu)$ agree.
\end{Theorem}

\section{The Joseph-Hotta construction of Springer representations and a conjecture about Weyl group actions in general}

Because  $H_c^\top( \overline{\Gr^\lambda} \cap S_0)$ is identified with $V(\lambda)_0$ under Geometric Satake, we obtain a $S_d$-action on $H_c^\top( \overline{\Gr^\lambda} \cap S_0)$. On the other hand we have a bijection between $ \overline{\Gr^\lambda} \cap S_0$  and and the irreducible components of any Springer fiber $\tcN^x_{1^d}$ where $x \in \OO_\lambda$. A priori there are two such bijections, but by Theorem \ref{thm:the-two-bijections-of-irreducible-components-agree} the two bijections agree. Therefore we also have a $S_d$-action on $H_c^\top( \overline{\Gr^\lambda} \cap S_0)$ coming from Springer theory. By Theorem \ref{thm:bgv-weight-zero-springer-action} these two actions differ by tensoring with the sign representation of $S_d$.

Joseph conjectured \cite{Joseph}, and Hotta \cite{Hotta} later proved that one can construct the Springer action directly on the span of orbital varieties in terms of equivariant multiplicities. The data of these equivariant multiplicities is recorded by the so called Joseph polynomials. This will allow us to rephrase Theorem \ref{thm:bgv-weight-zero-springer-action} in a way that makes sense for arbitrary types and arbitrary dominant coweights $\lambda$. Therefore we can state a general conjecture that our work shows is true in type A and for $\lambda \leq d \omega_1$.

\subsection{Phrasing Joseph-Hotta construction}
We will phrase Joseph-Hotta construction in a way that is will make sense for MV cycles in general. Recall that the Borel-Moore homology $H_\top( \overline{\OO_\lambda} \cap \fu)$ is the vector space dual of $H_c^\top( \overline{\OO_\lambda} \cap \fu)$. Furthermore, $H_\top( \overline{\OO_\lambda} \cap \fu)$ has a basis indexed by orbital varieties $\Irr( \overline{\OO_\lambda} \cap \fu)$. Each orbital variety is $T$-invariant and therefore has an equivariant fundamental class in  $H^T_\top( \overline{\OO_\lambda} \cap \fu)$.

We have a proper pushforward map
\begin{align}
  \label{eq:74}
H^T_\bullet(0) \rightarrow H^T_\bullet(\overline{\Gr^\lambda} \cap S_0)
\end{align}
that becomes an isomorphism after tensoring with the fraction field of $H^\bullet_T(\pt)$. In particular, for each irreducible component $Z \in \Irr( \overline{\OO_\lambda} \cap \fu)$, we can consider the $T$-equivariant multiplicity $e^T_0(Z)$ of $Z$ at the fixed point $0$.
\begin{Theorem}[\cite{Hotta}]
  \label{thm:joseph-hotta-construction}
  The map
  \begin{align}
    \label{eq:78}
    H_\top( \overline{\OO_\lambda} \cap \fu) \rightarrow H^\bullet_T(\pt)
  \end{align}
  defined by 
  \begin{align}
    \label{eq:79}
    [Z] \mapsto e^T_0(Z)
  \end{align}
  for each orbital variety $Z$ and extending linearly is injective. Furthermore, its image is a Weyl group invariant subspace of $H^\bullet_T(\pt)$.

Let $x \in \OO_\lambda$. Under the bijection
  \begin{align}
    \label{eq:80}
   s_\lambda: \Irr( \overline{\OO_\lambda} \cap \fu) \overset{\sim}{\rightarrow} \Irr(\tcN^x_{1^d})
  \end{align}
  the Weyl group action induced by \eqref{eq:78} action coincides with the Springer action on $H_\top(\tcN^x_{1^d})$ tensored with the sign character.
\end{Theorem}
\noindent We refer the reader to the book of Borho, Brylinski, and MacPherson \cite{Borho-Brylinski-MacPherson}, especially chapter 4, for a thorough description of this story.

\begin{Remark}
  
Because $\fu$ is smooth, one can write 
\begin{align}
  \label{eq:77}
  e^T_0(Z) = \frac{J_Z}{\eu_0(\fu)}
\end{align}
where $J_Z \in H^\bullet_T(\pt)$. In commutative algebra language, the polynomial $J_Z$ is essentially the $T$-equivariant multidegree of $Z$ embedded in $\fu$. In this specific setting it is called the \terminology{Joseph polynomial} of the orbital variety $Z$. Usually the above theorem is phrased in terms of Joseph polynomials. As $\eu_0(\fu)$ transforms under the sign character, the analogue of Theorem \ref{thm:joseph-hotta-construction} stated with Joseph polynomials in place of equivariant multiplicities does not require tensoring with sign representation.
\end{Remark}

\subsection{Rephrasing Theorem \ref{thm:bgv-weight-zero-springer-action} in terms of equivariant multiplicities}

Let $\bG$ be a reductive group as in \S . Let $\lambda$ be a dominant coweight that lies in the coroot lattice. We can consider the representation $V(\lambda)$ of the dual group $\bG^\vee$ and the zero weight space $V(\lambda)_0$. The Weyl groups of $\bG$ and $\bG^\vee$ are canonically identified; write $W$ for this group.

Then we have the identification:
\begin{align}
  \label{eq:113}
 V(\lambda)_0 \cong  H^\top_c(\overline{\Gr^\lambda} \cap S_0)
\end{align}
Therefore, we have a $W$-action on $H^\top_c(\overline{\Gr^\lambda} \cap S_0)$. Taking vector space dual we obtain a $W$-action on $H_\top(\overline{\Gr^\lambda} \cap S_0)$. 
We can now state the main conjecture of this paper.
\begin{Conjecture}
  \label{conj:equivariant-multiplicities-and-weight-zero}
  The map
  \begin{align}
    \label{eq:111}
    H_\top(\overline{\Gr^\lambda} \cap S_0) \rightarrow \FF(H^\bullet_T(\pt))
  \end{align}
  defined by 
  \begin{align}
    \label{eq:79}
    [Z] \mapsto e^T_0(Z)
  \end{align}
  for each MV cycle $Z$ and extending linearly
  is $W$-equivariant.
\end{Conjecture}

Combining Theorem \ref{thm:bgv-weight-zero-springer-action} and Theorem~\ref{thm:joseph-hotta-construction}, we obtain the following.
\begin{Theorem}
  \label{thm:main-theorem-in-text}
  Conjecture \ref{conj:equivariant-multiplicities-and-weight-zero} is true for $\SL_d$ and $\lambda \leq d \omega_1$.
\end{Theorem}

\bibliographystyle{amsalpha}
\bibliography{references}

\end{document}